\documentclass[reqno,12pt]{amsart}
\usepackage{a4wide}
\usepackage{geometry}
\usepackage{setspace}
\usepackage{amsmath, amssymb, latexsym, amsthm}
\usepackage{mathrsfs, mathtools, xparse}
\usepackage{hyperref}
\usepackage{xcolor}  

\usepackage{paralist}
\usepackage[T1]{fontenc}

\numberwithin{equation}{section}

\hypersetup{
    colorlinks=true, 
    linktoc=all,     
    linkcolor=blue,  
}

\theoremstyle{plain}
\newtheorem{thm}{Theorem}[section]
\newtheorem{lem}[thm]{Lemma}
\newtheorem{prop}[thm]{Proposition}
\newtheorem{cor}[thm]{Corollary}
\newcommand{\thmref}[1]{Theorem~\ref{#1}}
\newcommand{\lemref}[1]{Lemma~\ref{#1}}
\newcommand{\propref}[1]{Proposition~\ref{#1}}
\newcommand{\corref}[1]{Corollary~\ref{#1}}

\theoremstyle{definition}
\newtheorem{rmk}[thm]{Remark}

\newtheorem{conjecture}[thm]{Conjecture}

\newcommand{\rmkref}[1]{Remark~\ref{#1}}

\newcommand{\conjectureRef}[1]{Conjecture~\ref{#1}}

\newcommand{\psmb}{\left( \begin{smallmatrix}}
\newcommand{\psme}{ \end{smallmatrix} \right)}
\newcommand{\smat}[4]{\left( \begin{smallmatrix} #1 & #2 \\ #3 & #4 \\ \end{smallmatrix} \right)}

\DeclareMathOperator{\im}{Im}
\DeclareMathOperator{\tr}{tr}
\DeclareMathOperator{\rk}{rank}
\DeclareMathOperator{\Adj}{Adj}
\renewcommand*{\mod}{\operatorname{mod}}
\newcommand*{\halfSpace}{\mathcal H}
\newcommand{\z}{\mathbb{Z}}
\newcommand{\hn}{\mc H_n}
\newcommand{\lan}{\langle}
\newcommand{\ran}{\rangle}
\newcommand{\midmid}{\;\middle|\;}
\newcommand{\q}{\quad}
\newcommand{\fn}{\mc F_n}
\newcommand{\tg}{\tilde{\gamma}}
\newcommand{\tgg}{\tilde{\gamma_1}}

\newcommand{\prodd}{\prod \nolimits}
\newcommand{\summ}{\sum \nolimits}
\newcommand*{\Q}{\mathbb{Q}}
\newcommand*{\R}{\mathbb{R}}
\newcommand{\N}{\mbb{N}}
\newcommand*{\complex}{\mathbb{C}}
\newcommand*{\fundamentalDomain}{\mathcal{F}_n}

\newcommand{\mbb}{\mathbb}
\newcommand{\mc}{\mathcal}
\newcommand{\mrm}{\mathrm}
\newcommand{\tra}{\mrm{tr}}
\newcommand{\dia}{\mrm{diag}}
\newcommand{\n}{\nonumber}
\newcommand{\mf}{\mathbf}

\newcommand{\bnk}{\mc B_k^n}
\newcommand{\ank}[1][n]{a(#1, k)}
\newcommand*{\basis}[1][n]{\mc B_k^{#1}}
\newcommand*{\bkzz}[1][Z]{\mbb B_k(#1, #1)}
\newcommand{\snk}{S^n_k}

\newcommand{\dn}{\nonumber}

\newcommand{\cy}{\mc C_Y}
\newcommand{\ty}{\widetilde{Y}}
\newcommand{\tyd}{\ty_D}
\newcommand{\ymin}{\lambda_1}
\newcommand{\sn}{\sigma^{(n)}}
\newcommand{\ut}{\underline{t}}
\newcommand{\sumn}{\sum\nolimits}

\newcommand*{\GL}[2][n]{\operatorname{GL}(#1, #2)}
\newcommand*{\SL}[2][n]{\operatorname{SL}(#1, #2)}
\newcommand*{\Sp}[2][n]{\operatorname{Sp}(#1, #2)}
\newcommand*{\Sym}[2][n]{\operatorname{Sym}(#1, #2)}
\newcommand*{\SO}[2][n]{\operatorname{SO}(#1, #2)}
\newcommand*{\m}[2][n]{\operatorname{M}(#1, #2)}

\newcommand{\tp}[1]{#1^t}

\newcommand*{\expn}[1]{\exp\left(#1\right)} 

\newcommand*{\sptwo}{\Sp[2]{\R}}
\newcommand*{\mtwo}{\m[2]{\z}}

\newcommand*{\QEDB}{\hfill\ensuremath{\square}}

\newcommand*{\norm}[1]{\left\lVert#1\right\rVert}
\newcommand*{\abs}[1]{\left\lvert#1\right\rvert}

\newcommand{\lv}{\left|}
\newcommand{\rv}{\right|}
\NewDocumentCommand{\oldnorm}{sO{}m}{%
  {\IfBooleanTF{#1}
    {\oldnormaux{\lv}{\rv}{#3}}
    {\oldnormaux{#2|}{#2|}{#3}}}
}
\makeatletter
\newcommand{\oldnormaux}[3]{\mathpalette\oldnormaux@i{{#1}{#2}{#3}}}
\newcommand{\oldnormaux@i}[2]{\oldnormaux@ii#1#2}
\newcommand{\oldnormaux@ii}[4]{%
  \sbox\z@{$\m@th#1#2#4#3$}%
  \sbox\tw@{$\m@th\|$}%
  \mathopen{\hbox to\wd\tw@{\hss\vrule height \ht\z@ depth \dp\z@ width .3\wd\tw@\hss}}%
  #4
  \mathclose{\hbox to\wd\tw@{\hss\vrule height \ht\z@ depth \dp\z@ width .3\wd\tw@\hss}}%
} 
\makeatother

\newcommand{\on}{\oldnorm}

\author{Soumya Das}
\address{Department of Mathematics\\ 
Indian Institute of Science\\ 
Bengaluru -- 560012, India.}
\email{soumya@iisc.ac.in}

\author{Hariram Krishna}
\address{Department of Mathematics\\ 
Indian Institute of Science\\ 
Bengaluru -- 560012, India.}
\email{hariramk@iisc.ac.in} 

\date{}

\begin{document}
\title[Sup-norm of Siegel cusp forms]{Bounds for the Bergman kernel and the sup-norm of holomorphic Siegel cusp forms}

\begin{abstract}
We prove `polynomial in $k$' bounds on the size of the Bergman kernel for the space of holomorphic Siegel cusp forms of degree $n$ and weight $k$. When $n=1,2$ our bounds agree with the conjectural bounds on the aforementioned size, while the lower bounds match for all $n \ge 1$. 
For an $L^2$-normalised Siegel cusp form $F$ of degree $2$, our bound for its sup-norm is $O_\epsilon (k^{9/4+\epsilon})$.
Further, we show that in any compact set $\Omega$ (which does not depend on $k$) contained in the Siegel fundamental domain of $\Sp[2]\z$ on the Siegel upper half space,
the sup-norm of $F$ is $O_\Omega(k^{3/2 - \eta})$ for some $\eta>0$, going beyond the `generic' bound in this setting.

\end{abstract}

\maketitle

\section{Introduction}
This article is concerned with the problem of estimation of the sup-norm of an $L^2$-normalised holomorphic Siegel cusp form of degree $n$ (which is fixed but arbitrary)
in terms of its scalar weight $k$, which is assumed to be `large' as compared to $n$.
More precisely, let $F$ be a Siegel cusp form of degree $n$ and weight $k$ on the Siegel modular group $\Gamma_n:=\Sp{\z}$.
We denote the space of such forms by $S_k^n$.
Further assume that $\oldnorm{F}_2=1$, where $\oldnorm{F}_2$ denotes the Petersson norm of $F$.
Let $Z=X+iY \in \hn$, the Siegel's upper half-space. Then we define the sup-norm of $F$ (which is indeed finite, see e.g. \cite{klingen1990siegel}) as follows
\begin{equation} \label{hgamma-sum}
\on{F}_\infty :=\norm{\det(Y)^{k/2} F}_\infty = 
\sup_{ Z \in \hn}{
\det(Y)^{k/2}|F(Z)|}
.\end{equation}
By the invariance of $\det(Y)^{k/2}|F(Z)|$ under $\Sp{\z}$, we can restrict ourselves to a suitable fundamental domain.
The quest for suitable upper bounds on $\oldnorm{F}_\infty$ is the old and classical sup-norm problem,
which has its genesis in the question of understanding the mass distribution of the eigenfunctions
on a complete Riemannian manifold of dimension $d$ without boundary, pioneered by \cite{iwaniec1995supnorms}.
We are interested in arithmetic quotients of the form $X=L \backslash S$ where the symmetric space
$S=\hn$ (Siegel's upper half-space of degree $n$, see section~\ref{prelim}),
the arithmetic lattice is $L=\Gamma_n= \Sp \z$ and the group of isometries $G= \Sp \R/\pm$.
We will work in this non-compact setting and let $k \to \infty$.
However, we would also consider compact subsets of $X$ (see section~\ref{ampl}) and go beyond the expected `generic' bound.

Let us mention some results known in our setting.
For a holomorphic cuspidal Hecke eigenform $f$ on $\SL[2]{\z}$ of weight $k$,
H. Xia \cite{xia2007norms} has shown the sharp result
\begin{equation} \label{xia}
  k^{\frac{1}{4}-\epsilon}
  \ll_{\epsilon}\oldnorm{f}_\infty
  \ll_{\epsilon}k^{\frac{1}{4}+\epsilon}.
\end{equation}
The proof uses deep results on the bounds on $L(1, \mrm{sym}^2(f))$ and the Fourier expansion of $f$.
Among some other results available in the weight aspect in the holomorphic setting, we mention \cite{kramer, fjk2016, kramer1, blomer2015size, blomer2016supnorm, cogdell2011bergman, das2015supnorms, das-rms-supnorms}. In fact the results on Bergman kernel estimates of Kramer \cite{kramer, kramer1}, especially \cite{fjk2016} et. al. was one of our main motivations behind this article. Let $B^1_k(\Gamma)$ be an orthonormal basis for $S_k(\Gamma)$, where $\Gamma \subset \mrm{SL}(2,\R)$ is a Fuchsian subgroup of the first kind. Then more precisely it was shown in \cite{fjk2016} that for $k \ge 2$,
\begin{align*}
    \sup\nolimits_{z \in \mc H} \sum \nolimits_{f \in \mf B^1_k(\Gamma)} \im(z)^k |f(z)|^2 = O_\Gamma(k^{3/2}),
\end{align*}
whereas quite explicit estimates for the same when $\Gamma=\mrm{SL}(2,\z)$ were obtained in \cite{kramer1}. The main approach in these papers was the usage of the heat-kernel method.
However in this article we adopt a more direct, hands-on perspective on the same problem, using analytic methods.

In the context of $n=2$, \cite{blomer2015size} V. Blomer considered the case of the Saito-Kurokawa lifts (from elliptic cusp forms)
and based on the results obtained, speculated about the following conjecture on the size of Siegel Hecke eigenforms of degree $n$ and provided some support towards it.
\begin{conjecture} \label{bloconj}
Let $F \in S^n_k$ be an $L^2$-normalised Hecke eigenform. Then,
\begin{equation} \label{sklift}
\oldnorm{F}_\infty = k^{n(n+1)/8 +o(1)} \q \q (\text{as } k \to \infty).
\end{equation}
\end{conjecture}
This conjecture is consistent with the conjectural size of the Bergman kernel for $\Sp \z$, see  \conjectureRef{bergconj} below.
Assuming the Lindel\"{o}f hypothesis on all quadratic twists of a Hecke newform of weight $k$ on $\mrm{SL}(2, \z)$,
Blomer shows that the upper bound as in \eqref{sklift} holds when $n=2$. See \cite{das-rms-supnorms} for an alternative approach.
The aim of this paper is to establish for all $n \ge 1$ bounds for $\oldnorm{F}_\infty$ which are polynomial in $k$. Such kind of bounds should be expected from the generalities on the theory of automorphic forms.

When one moves to higher degrees, i.e., considers holomorphic modular forms on $\Sp{\z}$, the situation is far from understood.
The only result in the literature seems to be that of \cite{blomer2016supnorm}, where Siegel-Maa{\ss} Hecke eigenforms
of weight $0$ were considered in a fixed compact subset of $\Gamma_2 \backslash \mc H_2$
and a power saving was obtained over the `generic' bound, by the method of amplification.
One of the main difficulties in higher degrees, as also mentioned in the introduction to
\cite{blomer2015size}, is that we have very little control over the Fourier expansion of $F$,
even if it is a Hecke eigenform, as opposed to the case of $\mrm{GL}(n)$. So naturally, this article is heavily reliant on properties of the Bergman kernel, which is a powerful tool in studying the sup-norm problem.

Recall that one expression for the Bergman kernel $B_k(Z,W)$ ($Z,W \in \mathbb H_n$) is given by (we suppress the dependence on $n$) 
\[ B_k(Z,W)= \sum_{G \in \mathcal B^n_k} G(Z) \overline{ G(W)} \]
where $\basis$ denotes an orthonormal basis of $S_k^n$. More intrinsically, the Bergman kernel can be expressed as a certain average over $\Gamma_n$:
\begin{equation} \label{bergdef0}
 B_k(Z,W)=  \frac{a(n,k)}{2} \, \summ_{\gamma \in \Gamma_n} \det \Big(\frac{Z- \overline{W}}{2i}\Big)^{-k} \Big\vert_k^{(W)}\gamma \q \q (k \ge 2n+2) ;
\end{equation}
where $a(n,k)$ is a constant (cf.~\eqref{ank}). The Bergman kernel appears in the theory of automorphic functions in many connections, e.g., as a reproducing kernel for the space $S^n_k$ (cf.~\cite{klingen1990siegel}) and is used in the theory of metrization, for dimension formulae, as the ratio of the `canonical' and the hyperbolic metrics on the quotient space $\Gamma_n \backslash \hn$ (cf.~\cite{kramer3}) etc. Basic properties of $B(Z,W)$ have been studied for e.g. in \cite{cogdell2011bergman, klingen1990siegel}.

We now define one of the central objects of study in this paper:
\begin{equation} \label{bergdef1}
\bkzz := \sum\nolimits_{F\in\basis}{|F(Z)|^2\det(Y)^k} \q \q (= \det(Y)^k B_k(Z,Z)),
\end{equation}
and by abuse of notation continue to call this as the Bergman kernel as well, we believe that there will be no confusion, as the meaning will be clear.

From \eqref{bergdef1}, it is clear that $\bkzz$ is invariant under $\Sp \z$, is bounded on $\hn$, and that this bound should depend only on $n$ and $k$. Moreover, there are hints towards a conjecture (see \cite{anil, kramer, kramervid}, even though not mentioned explicitly anywhere) that $\mbb B_k$ satisfies the following  bound.
\begin{conjecture} \label{bergconj}
With the above notation and setting, the following is true.
\begin{equation} \label{bergconjeq}
 \sup \nolimits_{Z \in \hn} \bkzz \asymp_{n} k^{\frac{3n(n+1)}{4} }.
\end{equation}
\end{conjecture} 

Note that the exponent $\frac{3n(n+1)}{4}$ equals $3/2 \cdot \dim(\mc A_n)$, where $\dim(\mc A_n)$ is the complex dimension of the Siegel modular variety $\mc A_n=\Gamma_n \backslash \hn$. Usually, such bounds are obtained by a detailed analysis of the corresponding heat kernel (see e.g. \cite{kramer}).
Thus granting this, one would obtain, by dropping all but one term in \eqref{bergconjeq}, the following bound for the sup-norm of $F$:
$\oldnorm{F}_\infty \ll_n k^{\frac{3n(n+1)}{8}}$.

Notice here that the conjectural bound $\oldnorm{F}_\infty \ll_{n, \epsilon} k^{n(n+1)/8 +\epsilon}$ (cf.~\eqref{sklift}) leads to the upper bound in \eqref{bergconjeq} by summing over a basis of $S^n_k$.
One of our results in this article is the following bound  for the quantity $\bkzz$, with the aim of obtaining sup-norm bounds for individual cusp forms.
As far as we are aware,
this is the first result\footnote{A video lecture \cite{kramervid} outlining a bound on the Bergman kernel for $\Sp[2]{\z}$
(from which the upper bound in \eqref{supbds} when $n=2$ follows) is available; but not any pre-print, to our knowledge.
Anyway, the methods seem very different.} in the literature which treats the sup-norm problem for holomorphic Siegel modular forms in the weight aspect {\em uniformly in $Y$}, certainly when $n \ge 3$. In compact sets, the behaviour of $\bkzz$ is well-understood, see e.g. section~\ref{ampl}.
We obtain the following result.
\begin{thm} \label{mainthm}
Let $\epsilon>0$ be given. Put $\ell(n):=3n(n+1)/4$ and suppose $k$ is even. Then with the above notation and setting,
\begin{align}  \label{supbds}
  \left.\begin{array}{l} k^{\ell(1)} \\  k^{\ell(2)} \\ k^{\ell(n)}
  \end{array}\right\}
  \ll_n \sup_{Z\in\hn} \bkzz 
  \ll_{n,\epsilon} 
  \left\{\begin{array}{ll} k^{\ell(1)} \q & (n=1) \\
  k^{\ell(2) + \epsilon} \q & (n=2) \\
  k^{5\ell(n)/3 -3(n+1)/4 +\epsilon} \q  &(n \ge 3)  \\
  \end{array}\right..
\end{align}
\end{thm}
Therefore, when $n=1, 2$ the first two set of inequalities in \eqref{supbds} prove \conjectureRef{bergconj} (up to the $\epsilon$ in the upper bound when $n=2$). For the proof of this theorem, see section~\ref{mainthmproof}. Of course our bound for $n \ge 3$ is far away from the conjectured bound, but the reader should also take into account the severe paucity of available results when $n \ge 3$, which are otherwise deemed standard, e.g. the analysis of the Fourier expansion of Poincar\'e series or reasonable bounds on Fourier coefficients (say Deligne's bound for elliptic newforms). One novelty of this paper is to get around this point when $n \ge 3$ (cf. \lemref{beta0} and the discussion after \eqref{cnk}).
Towards the sup-norm, we thus have the following.
\begin{cor}
Let $F \in S^n_k$ be $L^2$ normalised. Then 
\begin{align} \oldnorm{F}_\infty \ll_{n,\epsilon}
    \begin{cases}
      k^{9/4 +\epsilon} & \q (n=2), \\
      k^{(5n-3)(n+1)/8 +\epsilon} & \q (n \ge 3).
    \end{cases}
\end{align}
\end{cor}

Note that $F$ above need not be a Hecke eigenform.
Since $\dim \snk \asymp k^{n(n+1)/2}$  (cf.~\cite{klingen1990siegel}), the bounds in \eqref{supbds} therefore show that 
for $k$ even and large enough, there exist $L^2$ normalized Hecke eigenform $F_1 \in S^n_k$ such that $\oldnorm{F_1}_\infty \gg k^{n(n+1)/8}$, if we choose the orthonormal basis to be consisting of Hecke eigenforms.
This slightly improves \cite[Theorem~1]{blomer2015size} where a lower bound $k^{n(n+1)/8 - \epsilon}$ was shown for some cusp form (arising as an Ikeda lift) of even degree $n$. Interestingly, the sizes of cusp forms can go at least up to the square-root of the size of the BK, as the corollary below shows.
\begin{cor}
There exist an $L^2$-normalised cusp form $G \in S^n_k$  (which is not an eigenform) such that $\norm{G}_\infty \gg k^{3n(n+1)/8}$.
\end{cor}
From \thmref{mainthm} (see also \eqref{localize}), there exists $Z_0 \in \hn$ such that $\mbb B_k(Z_0,Z_0) \gg
k^{3n(n+1)/4}$. Now let us define \footnote{We are grateful to G. Harcos for pointing this out.}
\begin{equation}
    G(Z) := {\mathbb{B}_k(Z_0,Z_0)}^{-1/2}
\sumn_{F \in \mc B_k^n} \overline{ F(Z_0) }  F(Z) \det(Y_0)^{k/2}
\end{equation}
Then $G$ is $L^2$-normalized, and its $L^\infty$-norm is at least
$ \det(Y_0)^{k/2}|G(Z_0)| = {\mathbb{B}_k(Z_0,Z_0)}^{1/2} \gg k^{3n(n+1)/8}$.

When $n=2$, put 
\[ \bkzz^\sharp := \summ_{F \in B^\sharp_k} \det(Y)^k |F(Z)|^2,\]
where $B^\sharp_k$ is a Hecke basis for $S^\sharp_k \subset S^2_k$, the orthogonal complement of the space spanned by the Saito-Kurokawa lifts.
Then the following corollary shows that the bulk of the contribution to $\bkzz$ comes from the non-lifts, as is expected. Let $k$ be even and large.
\begin{cor} \label{skcor}
One has $k^{9/2} \ll \sup_Z \bkzz^\sharp \ll_\epsilon k^{9/2+\epsilon}$. There exist $F_2 \in S^\sharp_k$ such that
 $\oldnorm{F_2}_\infty \gg k^{3/4}$.
\end{cor}
The upper bound for $\bkzz^\sharp$ is trivial. The lower bound follows from \eqref{supbds} by noting that the contribution from the Saito-Kurokawa (SK) lifts (consisting of Hecke eigenforms) is at most $k^{7/2 +\epsilon}$ as the sup-norm of such a lift can be bounded by $k^{5/4+\epsilon}$ 
For SK lifts ($k$ even), let us put $\bkzz^*$ for the Bergman kernel of that space.
We mention here that in the recent work \cite{das-anamby}
an \textsl{unconditional} bound $k^{5/2} \ll \sup_Z \bkzz^* \ll_\epsilon k^{5/2+\epsilon}$ has been established, from which follows the bound 
$\oldnorm{F}_\infty \ll k^{5/4 +\epsilon}$ for $F$ a SK lift.
Therefore we call this the 'trivial bound' for SK lifts of weight $k$. Previously there was no such unconditional result, as e.g. the works
\cite{das-rms-supnorms}, or (the result obtained by the method of) \cite{blomer2015size} all use the main result from \cite{young}, which at present holds only for all \textsl{odd} fundamental discriminants. Anyhow, the upper bound for $\bkzz^\sharp$ follows.
As far as we are aware this is the first time in the literature that the existence of a non-SK lift satisfying the lower bound in \conjectureRef{bloconj} is known.

Our next corollary concerns the spinor zeta function $L(s,F)$ of $F$. Let $F \in S^2_k$ be an $L^2$-normalised Hecke eigenform. By B\"ocherer's conjecture \cite{boechconj}, proved by Furusawa, and Morimoto \cite{fumo} (see also the discussion in \cite{blomer2019spectral})
one has the relation
\begin{align} \label{af12}
     a_F(1_2)^2 = \frac{ 256 \pi^{11/2} (4 \pi )^{2k-3} \Gamma(2k-4)} {\Gamma(k-3/2) \Gamma(k-2) \Gamma(2k-1) } 
     \frac{ L(1/2,F)L(1/2, F \otimes \chi_{-4}) }{L(1, \mrm{Ad}F)};
\end{align}
where $1_2 = \smat{1}{0}{0}{1}$ and $L(s, F \otimes \chi_{-4})$ denotes the spinor zeta function twisted by the unique quadratic character $\bmod{\,4}$. Asymptotics of spectrally weighted averages of related family of $L$-functions were obtained in \cite{blomer2019spectral}.
Our main theorem implies the following sharp result about Lindel\"of on average for the family of the products of the central $L$-values $L(1/2,F)L(1/2, F \otimes \chi_{-4})$. Let $L(s, \mrm{Ad}F)$ denote the degree~$10$ adjoint $L$-function of $F$. Let $H(2,k)$ be the set of $L^2$-normalised Hecke eigenforms in $S^2_k$. We appeal to \eqref{af12} the bounds on $L(1, \mrm{Ad}F)$ from \cite{xiannanli} to arrive at the following.

\begin{cor} \label{lindelofcor}
\begin{align}
k^3 \ll \sum \nolimits_{F \in H(2,k)} \frac{ L(1/2,F)L(1/2, F \otimes \chi_{-4}) }{L(1, \mrm{Ad}F)} \ll_\epsilon k^{3 +\epsilon} \label{lincor1} \\
k^{3-\epsilon}  \ll \sum \nolimits_{F \in H(2,k)}  L(1/2,F)L(1/2, F \otimes \chi_{-4}) \ll_\epsilon k^{3 +\epsilon}. \label{lincor2}
\end{align}
\end{cor}
The corollary above also follows from \cite[(1.8)]{blomer2019spectral}, but we want to highlight the use of the methods of this paper to obtain the such results. In particular, 
on \conjectureRef{bloconj} for $F$, we can show that one would get the Lindel\"of conjecture for $L(1/2,F)L(1/2, F \otimes \chi_{-4})$ by using \eqref{af12} -- see subsection~\ref{lincorpf} for this and the proof of the above corollary.

\subsection{Remarks and discussions about proofs} When $n=1$, \conjectureRef{bergconj} is known from the works \cite{kramer1, kramer3} by the heat-kernel method (also see \cite{anil} for the case of Hilbert modular forms), but it appears with $O_\epsilon(k^{3/2-\epsilon})$ as the lower bound. It also follows from \eqref{xia} by summing over a basis, with $k^{\pm \epsilon}$ defects in the upper and lower bound respectively. 
When $n=1$, our lower bound in \eqref{supbds} gives $k^{3/2}$, and in subsection~\ref{n=1exam}, we show how to obtain the upper bound in \eqref{bergconjeq} without any recourse to the techniques using the heat-kernel (as in \cite{kramer1}) or using the Hoffstein-Lockhart bound for the upper and lower bounds for the Fourier coefficient $a_f(1)$ of an $L^2$ normalised Hecke eigenform of level one (as in \cite{xia2007norms}). For the upper bound in \eqref{bergconjeq} we just use a uniform (with respect to $k$) Hecke-bound obtained by using trivial bounds for Kloosterman sums and Bessel functions; the purpose is to demonstrate that with such minimal hypotheses, the proof might generalise to higher degrees.

To prove \thmref{mainthm}, as is well-known, we will use two kinds of estimates:
one from the Fourier expansion of $F$, and another from the analysis of the Bergman kernel for $S^n_k$.
However, as mentioned above, due to the intractability of the Fourier expansion,
we essentially use it only to obtain the cut-off for the region
(roughly it is of the form $\{ Z \in \fn \mid \det(Y) \gg k^{n}\}$, with $\fn$ being the Siegel fundamental domain for $\Sp \z$)
in which the quantity $\bkzz$ decays exponentially (see \thmref{f-febd}). When $n=2$, however, the Fourier expansion can be handled somewhat better, see section~\ref{n=2case}.
In the complementary region, our primary tool would be the Bergman kernel, defined above. It will be clear that the Bergman kernel occupies a major portion of this paper.
In the following, we will briefly discuss the main points pertaining to the techniques leading to the each kind of estimates discussed above. 

We start with some highlights on the analysis of the Fourier expansion.
Suppose that we have an average bound for the Fourier coefficients of $F \in \bnk $ in the following form:
  \begin{equation} \label{fc1}
p_{n,k}(T):=  \summ_{F \in \bnk} |a_F(T)|^2  \ll_n  \big(
  \frac{ ( 4\pi)^{nk/2} k^{\alpha} \det(T)^{ k/2-\beta} }
  {\Gamma(k)^{n/2}   }  \big)^{2},
  \end{equation}
where $\alpha, \beta$ are real parameters. 
We comment here that it is in principle possible (at least when $k>2n$) to obtain values for $\alpha,\beta$
in terms of $n$ by estimating the Fourier coefficient $a_{P_T}(T)$ (expressed in terms of Bessel functions and Kloosterman sums of matrix index)
of the $T$-th Siegel Poincar\'e series $P_T:=P^n_{k,T}  \in S^n_k$,  from the Petersson trace formula
\begin{equation} \label{peter}
\summ_{G \in \mathcal B^n_k} |a_G(T)|^2 = c_{n,k}^{-1} \det(T)^{k-\frac{n+1}{2}} a_{P_T}(T) 
\end{equation}
with the constant 
\begin{equation} \label{cnk}
 c_{n,k} = \pi^{n(n-1)/2} (4 \pi)^{n(n+1)/4 - nk} \prodd_{j=1}^n \Gamma(k - (n+j)/2) .
\end{equation}
Unfortunately, very little is known about $\beta$ (which should depend only on $n$) -- we do not even know
if we can take $\beta \ge 0$ \textsl{with respect to this specific choice of dependence on $k$, this is crucial for us.} 
So we take a different route and first prove \lemref{beta0} towards this, in order to set the ball rolling. At present we do not know of any other means to obtain a result like \lemref{beta0}.
This is done using bounds for the Bergman kernel. However this only shows that $\beta \ge -\gamma_n $ for some quadratic polynomial $\gamma_n >0$ in $n$. However, we emphasise here that such a weak lower bound for $\beta$  \textsl{is sufficient} for the problem at hand, in the sense that it leads to the same bound as one would have obtained if we knew that $\beta = 0$ via the method of \lemref{beta0}, perhaps at the cost of $k$ being larger. In fact, we can show $\beta=0$ with the knowledge of \thmref{mainthm} (cf. \rmkref{overcount}~(2)). Any better result by this method is contingent on a value of $\beta$ being strictly positive. Moreover, as far as we could see, it seems that the only way to get $\beta > 0$ in our situation is to estimate directly $a_{P_T}(T)$, something which seems not possible at the moment due to insufficient knowledge of the matrix argument Bessel functions and Kloosterman sums. See \rmkref{expected-fcbd} for more on this matter.

Ignoring the dependence on $F$, one can indeed find a $\beta >0$ in \eqref{fc1},
as can be seen from a variety of ways, for instance by considering the Rankin-Selberg series
with respect to the Fourier coefficients, or by using Jacobi Poincar\'e series (see e.g. \cite{boecherer1993estimates},
\cite{boech-ragha}). However, finding a $\beta >0$ which does not go to
$0$ as $n \to \infty$ (even ignoring the dependence on $F$) is a well-known long-standing
open problem (see \cite{boecherer1993estimates} for the best-known result to date). 

The other (i.e., apart from using \eqref{peter}) known methods for estimation of Fourier coefficients of Siegel cusp forms seem not suitable for us, as the implied constants in these methods cannot be controlled satisfactorily. 
When one looks at the Fourier coefficients on average, one can work with them. Thus
our artefact for working with the Fourier series is the quantity
$p_{n,k}(T)$ (cf. \eqref{fc1}) for which \lemref{beta0} provides a bound. The idea then is to link this information with the Fourier expansions of the constituents $F$ of $\bnk$. We do this in \lemref{pt} by employing the Cauchy-Schwartz inequality judiciously and reduce our attention to a smooth function
$q_k(Y)$ by showing that $\bkzz \le q_k(Y)^2$; where
$q_k(Y)=q_{n,k}(Y) := \sum_{T \in \Lambda_n} p(T)^{1/2} \, \det(Y)^{k/2} \, \exp(- 2 \pi \tr(TY))$.

The rest of the work on Fourier series is then to bound $q_k(Y)$, and this is the content of \propref{p-febd}, which is at the heart of the matters. 
For the pair $(\alpha,\beta)$ as in \eqref{fc1}, the Fourier expansion of $F$ yields for all $Z \in \fn$
\begin{align} \label{fc0}
  q_k(Y) \ll_n
        \big(k^{n} \det(Y)^{-1}\big)^{(n+1)/4 -\beta}  \cdot k^{\frac{n}{4}+\alpha+\epsilon};
\end{align}
We show that only finitely many $T$'s contribute significantly to the size of $q_k(Y)$, and estimate the size of this set in \lemref{count-bd2},
from which \eqref{fc0} follows.
We also have a version of \eqref{fc0} in \propref{p-febd} where we only assume that $Y$ is Minkowski-reduced.

We would only use \eqref{fc0} when $\det(Y) \gg k^n$ or more generally when $y_n \gg k$ with $y_n$ being the largest diagonal element of $Y$. In such a region, one gets the upper bound in \thmref{mainthm} by appropriate values of $(\alpha, \beta)$.
By a careful choice of constant $\mf c_n$, one can also prove that in a region $\det(Y) \ge \mf c_n k^n$, $q_k(Y)$ decays exponentially, see \lemref{expdecay-cutoff}.
\eqref{fc0} can be improved as soon as we get better values of the pair $(\alpha, \beta)$ in \eqref{fc1}, this is discussed in \rmkref{expected-fcbd}.

Moreover note that by \cite[Hilfssatz~3.11]{freitag1983siegel}, for any $F \in \snk$, the quantity $\det(Y)^{k/2}|F(Z)|$ attains its supremum at a point in $\fn$. Our results (cf. \propref{f-febd}) suggest that this is attained `high in the cusp', at least in the average.
This can be compared with the work of Brumley-Templier \cite{brutemp}. From \thmref{mainthm}, we see that the supremum of $\bkzz$ grows as at least $k^{3n(n+1)/4}$, \thmref{f-febd} and the second bound in \thmref{bkbd} show that such a value can only be attained in the following region:
\begin{equation} \label{localize}
    \bkzz[Z_0]=\sup \nolimits_{Z\in\fn} \bkzz \implies k^{n/3} \ll_n \det(\Im({Z_0})) \ll_n k^n
\end{equation}
This is also supported by the results of section~\ref{ampl}.

When $n=2$ (discussed in sections~\ref{n2-fouriercoeffs},~\ref{n=2case}), we have a somewhat better handle on the Fourier expansion. In this case instead of \lemref{beta0}, we can rely on the Fourier coefficients of Poincar\'e series directly to bound $p_{2,k}(T)=\sum_F |a_F(T)|^2$; these have been worked out in a remarkable paper of Kitaoka \cite{kitaoka1984fourier}.

Lastly, the bound $\sup_{Z \in \mc H_2} \bkzz \ll_\epsilon k^{9/2+\epsilon}$ deserves special mention. The reason behind this better bound is that on the one hand we adapt, and estimate $\#\cy$ (the finite set $\cy$ being the effective support of the Fourier expansion) more precisely using the results in \cite{blomer2015size} stemming from the fact that when $n=2,$ one can express the eigenvalues of a matrix in terms of its entries. On the other hand, we crucially make use of the decay (with respect to $k$) of the `off-diagonal' terms of the expression for the $T$-th Fourier coefficient of the Poincar\'e series $P_T$. A key idea that we have used while bounding the $T$-th Fourier coefficient of Poincar\'e series $P_T$ is to keep two natural terms which arise in this course to be treated separately rather than merge them into a uniform single estimate. The resulting bound then looks like
\begin{equation}
    p_{2,k}(T) \ll_n c_{2, k}^{-1}\det(T)^{k}
    (k^{\alpha_1} \det(T)^{-\beta_1} + k^{\alpha_2} \det(T)^{-\beta_2})
\end{equation}
where $\alpha_j,\beta_j$ are certain explicit constants. This creates the window for obtaining robust bounds when interacted with the bound on $\#\cy$.
See \lemref{l-n2fourierbound}.
These, when combined with our Bergman kernel estimates (cf. \thmref{bkbd}), yields the desired bound after a delicate analysis, all of these can be found in section~\ref{n=2case} yielding \thmref{9/4}.

Let us now highlight some of the notable points in the Bergman kernel analysis. First of all let us mention that here the situation is somewhat better than that with the Fourier expansion.

We break the sum in \eqref{bergdef0} into three pieces by writing any element $\gamma\in\Gamma_n$ as
$\gamma=\smat{1_n}{S}{0}{1_n}\cdot\allowbreak\smat{U}{0}{0}{(\tp{U})^{-1}}\cdot\allowbreak\smat{*}{*}{C}{D}$, where the notation is explained below.
The first sum is over $\Sym \z$ which is handled by the Lipschitz formula. Here we present a new method of bounding the sum $\summ_{S\in\Sym{\z}} \det(Z+S)^{-k}$
through its Fourier expansion (see \lemref{lipfo}):
\begin{equation}
    \sum_{S\in\Sym{\z}} \det(Z+S)^{-k} \ll_n  \begin{cases}
    k^{\frac{n(n+1)}{4}} \det(Y)^{-k} \q \q &(Z \in \hn),\\
    \det(Y)^{-k +\frac{n+1}{4}} \q \q &(Z \in \fn).
  \end{cases}
\end{equation}
The second sum $\sum_{U}\det(Y+U\ty \tp{U})^{-k}$ over $\GL \z$, is quite subtle (see \lemref{headache}), and the third one is over the set of all in-equivalent coprime symmetric pairs $\{C,D\}$ (see section~\ref{bksec1}). Coincidentally the bounds ensuing from the first and second summations are just enough to handle the third, which is a majorant of the classical Siegel Eisenstein series: $\sum_{\{C,D\}} |\det(CZ+D)|^{-\rho}$ for suitable $\rho>0$, one also has to check that the bound is independent of $Y$. The results are summarised in \thmref{bkbd}. We obtain the following bounds towards the Bergman kernel for all $Z \in \mc F_n$:
\begin{equation} \label{berg1}
\bkzz=B_k(Z,Z) \det(Y)^k \ll_{n, \epsilon}
k^{\frac{n(n+1)}{2}}(\det(Y)/y_1^n)^{\frac{n+1}{2}+\epsilon}
\min \left\{ k^{ \frac{n(n+1)}{4} } , \,\det(Y)^{\frac{n+1}{4}} \right\}.
\end{equation}
Here $y_1$ is the smallest diagonal entry of $Y$. Expression \eqref{berg1} could perhaps be improved if we do not forgo cancellations as much as we did, but it is a non-trivial prospect. However, a promising approach worth investigating seems to be the generalisation of section~\ref{n=1exam} revolving around the setting of amplification -- see the discussion there. Some calculations towards this effect can already be found in section~\ref{ampl} when we deal with amplification.
This approach, and the treatment of the level aspect of the same problem, will be the content of a forthcoming paper by the authors.

We end the paper by proving a bound for the sup-norm of a Hecke eigenform $F$ when it is restricted to a fixed compact subset, say $\Omega \subset \mc F_n$ which does not depend on $k$. Here the preliminary bound is $O_\Omega(k^{n(n+1)/2})$, which follows from the work of Cogdell-Luo \cite{cogdell2011bergman}. When $n=2$, we go beyond this bound by employing the amplification method: based on previous works by \cite{blomer2016supnorm},  \cite{cogdell2011bergman}, \cite{das2015supnorms} and \cite{iwaniec1995supnorms}. 
The construction of the amplifier follows that in \cite{blomer2016supnorm}, but proceeding further one needs to count the number of double cosets in the similitude Hecke operators $T_m$, see \lemref{coset-count}.
We show that the bound can be improved to $O_\Omega(k^{3/2 - \eta})$ for some absolute constant $\eta>0$. See subsection~\ref{ampl}, \thmref{amplithm}. Some of the calculations here might be useful in considering the same problem in the non-compact setting.

Lastly let us mention that for $a,b \ge 0$, on compact subsets of the form $\mc F_{a,b}= \mc F_{a,b;n}:= \{ Z \in \fn \mid k^a \ll \det(Y) \ll k^b \}$, $\bkzz$ can be bounded by $k^{(n+1)(2n+3b)/4 +\epsilon}$ when $b\le n-1$, and by $\expn{-k^{a/n}}$ when $a > n$.

It is very interesting to investigate this problem in the level aspect, work on this is under progress by the authors.

\subsection*{Acknowledgements}
We thank V. Blomer, G. Harcos for valuable comments on the paper.
S.D. thanks IISc. Bangalore, UGC Centre for Advanced Studies and DST India for financial support.
H.K. acknowledges the support by a scholarship ((2/39(2)/2016/NBHM/R\&D-II/11410)) from the National Board for Higher Mathematics.

\section{Notation and setting} \label{prelim}

In this paper, we will mostly use standard notation, some of which are collected below, and the rest will be introduced as and when necessary. For standard facts about Siegel modular forms, we refer the reader to \cite{freitag1983siegel}, \cite{klingen1990siegel}.

\begin{inparaenum}[(1)]

\item
We use the standard conventions in analytic number theory. We note $A \ll B$ and $A=O(B)$ are the Vinogradov and Landau notations respectively. By $A\asymp B$ we mean that there exists a constant $c\ge 1$ such that $B/c\le A\le cB$. Any subscripts under them (e.g. $A\ll_n B$) indicates the dependence of any implicit constants on those parameters. Throughout, $\epsilon$ will denote a small positive number which can vary from one line to another.

\item 
$\z, \Q, \R,$ and $\complex$ denote the integers,
rationals, reals and complex numbers respectively. For a commutative ring $R$ with unit, $M(n, R)$ denotes the set of $n \times n$ matrices over $R$. $\GL[n]{R}$ denotes the group of invertible elements in $M(n, R)$. We will denote the transpose of $A$ by $\tp{A}$. Moreover, $1_n$ and $0_n$ will be the identity and zero matrices respectively.
      
\item 
 $\Sym{R}$ is the set of symmetric matrices with entries in $R$. For a square matrix $A$, we denote by $A_D$ the diagonal matrix with the diagonal elements of $A$ on the diagonal. For a tuple $(a_1,a_2,\ldots,a_n) \in R^n$, we put $\dia(a_1,a_2,\ldots,a_n)$ to be the diagonal matrix with these elements on the diagonal.
 
 From now on $R \subset \R$. We denote by $\Sym[n]{R}^+$ to be the cone of positive-definite matrices over $R$.

\item 
$\mrm{GSp}^+(n,\R)\subseteq\GL[2n]{\R}$ is the symplectic similitude group over $\R$. That is, we put $\mrm{GSp}^+(n,\R) = \left\{ M=\psmb A& B \\ C & D \psme \in \GL[2n]{\R} \mid \tp{M} \psmb 0_n & -1_n \\ 1_n & 0_n \psme M = \mu(M) \psmb 0_n & -1_n \\ 1_n & 0_n \psme \right\}$, $\mu(M)>0$. Then $\Sp[n]{R}$ is the subgroup for which $\mu(M)=1$ for all $M$. In particular, $\Gamma_n=\Sp{\z}$ is the Siegel modular group. Throughout the paper, we would use the above block representation for $M \in \Sp[n] {R}$; with $A,B,C,D$ having size $n$. Equivalently $M$ satisfies the relations $A \tp{D} - B \tp{C}=1_n,\, A\tp{B}=B\tp{A}, \, C \tp{D}=D\tp{C}$.
      
\item 
$\halfSpace_n:=\{X+iY\in\Sym{\complex}\mid Y>0\}$ is the Siegel half-space. $M \in \Sp \R$ acts on $\hn$ by $Z \mapsto (AZ+B)(CZ+D)^{-1}$. Moreover, for $M \in \mrm{GSp}^+(n, \R)$ and $F$ holomorphic on $\hn$ we put $F|_k M := \det(CZ+D)^{-k} F(M\langle Z \rangle)$ where $\det(M)^{-1/2} M = \psmb * & * \\ C & D \psme$. When $F$ depends on more than one variable like in $F(Z, W)$, we will use $|_k^{(W)}$ to indicate that the action is to performed for the $W$ variable.

We denote by $Y=\Im(Z) >0$ to denote the imaginary part of $Z \in \hn$, and often use the formula $\Im(M \lan Z \ran)= {\tp{(CZ+D)}}^{-1}Y\overline{(CZ+D)}^{-1}$.

\item 
$S_k^n$ is the space of Siegel cusp forms of weight $k$ on $\Gamma_n$, i.e., those which are holomorphic and satisfy $f |_k \gamma = F$ for all $\gamma \in \Gamma_n$, with the usual boundedness condition when $n=1$.
We will use $\basis$ to denote an orthonormal basis for this with $F_{k, j}$ as its elements.
      
\item 
$\Lambda_n:=\{T\in\Sym{\Q}\mid T>0; 2T=(2t_{ij}) \in\Sym{\z}; t_{ii} \in\z\}$. Throughout the paper, for two symmetric matrices $A,B$ we write $A>B$ to denote $A-B$ is positive definite.
Note that $\GL \z$ acts on $\Lambda_n$ from the left by $T \mapsto U T\tp{U}$.
      
\item 
The Fourier expansion of $F\in S_k^n$ is written as
$F(Z)=\sum_{T\in \Lambda_n} a_F(T) e(\tra(TZ))$
and $a_F(T) \in \complex$ are its Fourier coefficients. Further $e(\tra(A))=e^{2\pi i\tra(A)}.$

\item 
$\fn$ will denote the Siegel fundamental domain of
$\halfSpace_n$ under the action of $\Gamma_n$.
If $Z=X+iY\in\fn$, then the coordinates of $X=(x_{ij})$ satisfy
$-\frac{1}{2}\leq x_{ij} \leq \frac{1}{2}$; moreover, $Y\in R_n$ the Minkowski-reduced domain,
and $|\det(CZ+D)|\geq 1$ for all $\{C, D\}$ being lower blocks of elements in $\Gamma_n$. Note that there exist constants $r_n\ge 1$ such that for a Minkowski-reduced matrix $A$ of size $n$, we have (cf.\cite[p. 20 Lemma 2]{klingen1990siegel})
\begin{equation} \label{reduct-implication}
A_D/r_n \le A \le r_n A_D.
\end{equation}
Moreover if $Z \in \fn$, then $Y_D \gg_n 1_n$ (cf.\cite[p. 30 Lemma 2]{klingen1990siegel}).

For $F,G \in S^n_k$ we define the Petersson inner product $\langle F, G \rangle$ of $F$ and $G$ by 
\[ \lan F, G \ran = \int_{\fn} F(Z)\overline{G(Z)} \det(Y)^k \frac{ dX \, dY}{\det(Y)^{n+1}}.\]

\item
For each $T \in \Lambda_n$ we define the $T$-th Poincar\'e series $P^n_{k,T}$ by the infinite series $P^n_{k,T} = \sum_{\gamma \in \Gamma_\infty \backslash \Gamma_n}\allowbreak e(\tra(TZ)) |_k \gamma$ for $k>n+1$; where $\Gamma_\infty = \{ \psmb 1_n & S \\ 0_n & 1_n \psme \mid S \in \Sym \z \}$ is the subgroup of $\Gamma_n$ consisting of translations, and $\Gamma_{0,n}$ denotes the Siegel parabolic subgroup defined by $\{\gamma \in \Gamma_n \mid C=0 \}$. The most important property of $P_T$ is that it is dual to the Fourier coefficients $A_F(T)$ for all $F \in S^n_k$, and this results in the very useful Petersson's formula enunciated in \eqref{peter}.

\item
If $\gamma=\smat{*}{*}{C}{D} \in \Gamma_n$, then $\{C,D\}$ is called a coprime symmetric pair. Conversely, such a pair satisfies $C\tp{D}=D\tp{C}$ and can be completed to an element of $\Gamma_n$. Moreover, $\GL \z$ acts on such pairs from the left by left multiplication -- the representatives of the distinct equivalence classes are called `in-equivalent'.

\end{inparaenum}
\raggedbottom
\section{Proof of \thmref{mainthm} and \corref{lindelofcor}} \label{mainthmproof}
In this section we prove \thmref{mainthm} by invoking suitable results from later parts of the paper. We feel this allows for an easier presentation of the proof.

\subsection{Lower bound for the Bergman kernel} \label{bklbd}
We first prove the lower bounds in \thmref{mainthm}.
Note that the lower bound for claimed for $\bkzz$ is uniformly $O(k^{3n(n+1)/4})$ for all $n \ge 1$.
We start with the observation that for any fixed $Y_0>0$,
\begin{align}
\sup_{Z' \in \fn} \sum_{F \in \bnk} \det(Y')^{k} \abs{F(Z')}^2 
& \ge \int_{X \bmod 1} \, \sum_{F \in \bnk} \det(Y_0)^{k} \abs{F(X+iY_0)}^2 \, dX \n\\
& \ge \det(Y_0)^{k} \sum_{F \in \bnk} \, \abs{ \int_{X \bmod 1} \frac{F(X+iY_0)}{e(\tr(X+iY_0))}\, dX }^2 \expn{-4\pi \tr(Y_0)}\n \\
& = \det(Y_0)^{k} \expn{\tr(-4 \pi Y_0)} \sum_{F \in \bnk} \, |a_F(1_n)|^2 \label{lbdp1}
.\end{align}
As a consequence of the Petersson formula (see \eqref{peter}, \eqref{cnk}), we know that
\begin{align} \label{lbdp2}
\sum_{F \in \bnk} \, |a_F(1_n)|^2 = c_{n,k}^{-1} \, a_{1_n}(1_n), \q\q \q \q (c_{n,k} \text{ as in } \eqref{cnk})
\end{align}
where $a_{1_n}(1_n)$ denotes the $1_n$-th Fourier coefficient of the Poincar\'e series $P^n_{k, 1_n}$. From \cite[Thm.~4]{kowalski2011note} we see that for all even $k \ge k_0$ with $k_0$ depending only on $n$,  
\begin{equation} \label{pnonv}
a_{1_n}(1_n) \ge \frac{\# \mrm{Aut}(1_n)}{4}, 
\end{equation}
where $\mrm{Aut}(1_n) = \# \{ U \in \GL \z \mid \tp{U}U=1_n \} \gg 1$. Of course, \eqref{pnonv} is related to the difficult problem of non-vanishing of Poincar\'e series -- see \cite{das2012nonvanishingsiegelpoincare} for the case when $T$ varies and $k$ is fixed but large.

Putting together \eqref{lbdp1}, \eqref{lbdp2} and \eqref{pnonv}, and taking $Y_0:= k/(4 \pi )\cdot 1_n$ we see that 
\begin{align}
\sup_{Z' \in \fn} \sum_{F \in \bnk} \det(Y')^{k} |F(Z')|^2 &\gg \left( \frac{k }{4 \pi e} \right)^{nk}  c_{n,k}^{-1} 
 \gg_{n}  \frac{\Gamma(k)^n k^{n/2}}{\prod_{\nu=1}^n \Gamma(k - (n+\nu)/2)} \dn\\
 & \gg_{n} k^{n/2 + (3n^2 + n)/4} = k^{3n(n+1)/4}. \label{bklbdnew}
\end{align}
In the above, $k_0$ may not be effective, as we rely on the result from \cite{kowalski2011note}. However, when $n=2$, we can use $k_0=6$ from \cite[Proposition 3.3]{kowalski2012local} or from \cite[(1.8)]{blomer2019spectral}.
This proves the lower bound for all $n \ge 1$.

\subsection{Upper bounds for the Bergman kernel}
For the upper bound of $\bkzz$ in \thmref{mainthm} when $n=2$, we refer the reader to \thmref{9/4}.

For the cases $n\ge 3$, the upper bound follows by combining \thmref{f-febd} and \thmref{bkbd}. We invoke \thmref{f-febd} in the region $\{Z \in \fn \mid y_n \ge nk\,r_n/(2\pi) \}$ ($r_n$ as in \eqref{reduct-implication}) to get exponential decay. In the complementary region of $\fn$ all diagonal entries are $O_n(k)$. Here, we invoke the second bound in \thmref{bkbd} viz. $ \bkzz \ll_n k^{n(n+1)/2}\det(Y)^{3(n+1)/4+\epsilon}y_1^{-n(n+1)/2-n\epsilon}$. We note that this quantity takes the maximum value when the smallest diagonal entry (or essentially eigenvalue by reduction criteria) is small i.e. $O(1)$ and all other diagonal entries are equal to $O_n(k)$. So the maximum for the bound is attained when $\det(Y) = O_n(k^{n-1})$ giving us $\bkzz \ll_{n, \epsilon}k^{(5n-3)(n+1)/4+\epsilon}$.
We refer the reader to section~\ref{n=1exam} for the proof of the upper bound when $n=1$. \QEDB

\subsection{Proof of \corref{lindelofcor}} \label{lincorpf}
From \eqref{lbdp1} and \eqref{af12}, we can write after a rearrangement that
\begin{align}
 \frac{ (4 \pi )^{2k} } {k^3 \, \Gamma(k-3/2) \Gamma(k-2)  } \,
    \sum_{F \in H(2,k)} \frac{ L(1/2,F)L(1/2, F \otimes \chi_{-4}) }{L(1, \mrm{Ad}F)} = \sum_{F \in H(2,k)} a_F(1_2)^2 \\
     \ll  \det(Y_0)^{-k} \exp(\tr (4 \pi Y_0)) \sup_{Z \in \mc H_2} \bkzz.
\end{align}
We again choose $Y_0 = (k/4 \pi) 1_2$ and use Stirling's approximation to get
\begin{align}
 \sum_{F \in H(2,k)}   \frac{ L(1/2,F)L(1/2, F \otimes \chi_{-4}) }{k^3 \, L(1, \mrm{Ad}F)} \ll_\epsilon  k^{2k-9/2} \cdot k^{-2k} \cdot k^{9/2+\epsilon}; \label{lindavg}
\end{align}
from which \eqref{lincor1} in \corref{lindelofcor} easily follows. The lower bound follows from the lower bound of $\sum_{F \in H(2,k)} a_F(1_2)^2 $ from \eqref{pnonv} and the same calculations as shown above. We leave the details to the reader.

The second inequality \eqref{lincor2} then follows from \eqref{lincor1} by using the following bounds for $L(1, \mrm{Ad}F)$ (cf. \cite{xiannanli}):  $k^{-\epsilon} \ll_\epsilon L(1, \mrm{Ad}F) \ll_\epsilon k^\epsilon $.

The convexity bound for the product $L(1/2,F)L(1/2, F \otimes \chi_{-4})$ is seen to be $O(k)$. So \corref{lindelofcor} gives worse than convexity bound for individual terms (by non-negativity of the $L$-values). However if one used the obvious analogue of \eqref{lbdp1} for a single $F$, and applied the conjectural bound on $\oldnorm{F}_\infty$, one would obtain the Lindel\"of conjecture for the product of the central $L$-values under consideration!
This is evident from \eqref{lindavg}, where now the sum over $F$ is replaced by a single $F$, and the last term on the r.h.s. of \eqref{lindavg} by $k^{3/2+\epsilon}$.

\section{Bounding the sup-norm using the Fourier Series}
In this section, we will obtain a bound on the size of a Siegel cusp form via its Fourier expansion by using certain bounds on its Fourier coefficients. The method rests on the observation that in the Fourier expansion of $F$, only certain finitely many $T$ contribute to its mass. 

Recall the explicit bound on the Fourier coefficients from \eqref{fc1} with a saving $\beta$ over the Hecke's bound. This saving cannot be completely arbitrary, as follows from below. Of course, one expects $\beta \ge 0$, but we could not find a proof.
As far as we are aware, the lemma below is not available in the literature if the uniformity with respect to $k$ is requested.

\begin{lem} \label{beta0}
Let $n \ge 2$ and put $\gamma_n:= (n+1)(2n-3) /4$.
Then with the above notation, \eqref{fc1} is satisfied by the Fourier coefficients of any Siegel cusp form of unit Petersson norm for the pair $(\alpha,\beta)$ such that 
\begin{equation}
     \alpha = (5n^2 +3n)/8; \q   \, \beta= - \gamma_n   
\q  \q (k \ge 2(n+1))
.\dn\end{equation}
\end{lem}

\begin{rmk}
The sharp upper bound $\beta \le (n+1)/4 +\epsilon$ has been shown in \cite[Remark~5.3]{boecherer2014characterization}
  by using the Rankin-Selberg $L$-series $R_F(s) = \sum_{T \in \GL \z \backslash \Lambda_n} \frac{ |a_F(T)|^2 }{ \varepsilon(T) \det(T)^{s} }$ attached to $F$. Here $\varepsilon(T)$ denote the number of units of $T$.
  Namely if one has $a_F(T) = O_{k, \epsilon}(\det(T)^{k/2-\beta + \epsilon})$ for all $\epsilon >0$,
  then one must have $k-2 \beta +(n+1)/2 > k$ as $R_F(s)$
  has a pole at $s=k$. That the bound is sharp follows from \cite[Theorem~1.2]{das-niser-conf}.
\end{rmk}

\begin{proof}
  A value of $\beta$ could in principle be obtained by bounding the
  $T$-th Fourier coefficient of the Poincar\'e series. Unfortunately, this
  is apparently very hard -- as very little is known about 
  the underlying Bessel functions of order $\ge 3$.
  Instead, we take a different route and employ the Bergman kernel for this purpose.
  First note that
  \begin{align}
    \summ_{T\in\Lambda_n} |a_F(T)|^2 \expn{-4\pi\tr(TY)}
    = \int_{X \bmod 1} |F(Z)|^2 dX,
  \dn\end{align}
  and therefore, by summing over any orthonormal basis $\basis$ of $S^n_k$,
  we get
  \begin{align} \label{fbdbk}
    \summ_{F\in\basis} \, \summ_{T\in\Lambda_n} \det(Y)^k
    |a_F(T)|^2 \expn{-4\pi\tr(TY)}
    = \int_{X \bmod 1} \summ_F \det(Y)^k |F(Z)|^2 dX.
  \end{align}
  The integrand on the r.h.s. of \eqref{fbdbk} is the Bergman kernel,
  and in the notation of this paper (see \eqref{berg1}), equals $ \bkzz$.
  From \thmref{bkbd}, one of the bounds for $\bkzz$ reads $ \bkzz \ll_{n,\epsilon} k^{3n(n+1)/4} \det(Y)^{(n+1)/2 + \epsilon}$
  for all $Z \in \fn$. We will use this bound.

Let $Z_0=X_0+iY_0 \in \hn$, we can find an
 $M=\psmb A & B \\ C & D \psme \in \Gamma_n$
 such that $Z=M\langle Z_0 \rangle \in \fn$.  From the above mentioned bound for the Bergman kernel thus,
  \begin{align} \label{fest1}
  \bkzz[Z_0]= \bkzz
  \ll_n k^{3n(n+1)/4} \det(\Im (M\langle Z_0 \rangle))^{(n+1)/2+\epsilon}
  = \frac{k^{3n(n+1)/4} \, \det(Y_0)^{(n+1)/2+\epsilon}}{|\det(CZ_0+D)|^{n+1+\epsilon}}.
  \end{align}

  Let $\mrm{rank}(C)=r$. Then from Siegel \cite{siegel1935analytic} (see also \cite[Lemma~3.1]{das2015nonvanishingkoechermaass})
  we can  find $U,W \in \GL{\z}$ and $C_1, D_1 \in \m[r]{\z}$ such that
  \begin{align}
    (UC, UD) = \left(
    \begin{pmatrix} C_1 & 0 \\ 0 & 0 \end{pmatrix}\tp{W} ,
    \begin{pmatrix} D_1 & 0 \\ 0 & I_{n-r}\end{pmatrix} W^{-1}
    \right). \q \q (\mrm{rank}(C_1)=r)
  \dn\end{align}
  Using this we see that 
  \begin{align}
    |\det(CZ_0+D)| &= \left| \det \left(
    \begin{pmatrix}C_1 & 0 \\ 0 & 0\end{pmatrix} Z_0[W]
    + \begin{pmatrix}D_1 & 0 \\ 0 & I_{n-r} \end{pmatrix}
    \right) \right| \dn\\
    &= \left|\det(C_1 Z_0[W]_* +D_1)\right|,\n
  \end{align}
  where $Z_0[W]_*$ is the leading $r \times r$ sub-matrix of $Z_0[W]$.

Now take any $T_0 \in \Lambda_n$.
  Dropping all but one term in \eqref{fbdbk} we get from \eqref{fest1} that $ \sum_{F\in\basis}|a_{F}(T_0)|^2$ is
  \begin{align*} 
    &\ll_n \frac{k^{3n(n+1)/4} \expn{4\pi \tr(T_0 Y_0)} \det(Y_0)^{-k+ (n+1)/2 +\epsilon}}
        {|\det(C_1 Z_0[W]_*+D_1)|^{n+1+\epsilon}}.
    \end{align*}    
Note the inequality $|\det(C_1 Z_0[W]_*+D_1)| \ge \det(\Im(Z_0[W]_*)) = \det(Y_0[W]_*)$, since $\Im(A_*) = \Im(A)_*$ for any $A \in \hn$.    
   
  We choose $Y_0 = \frac{k}{4 \pi} \cdot T_0^{-1}$
  (the choice of the constant is dictated by the exponential term above)
  and then see from the above that
  \begin{align}
   \sum_{F\in\basis}|a_{F}(T_0)|^2     & \ll_n \frac{ k^{5n(n+1)/4-nk+\epsilon} \expn{nk} (4 \pi)^{nk}
                 \det(T_0)^{k-(n+1)/2 +\epsilon}}
               { \det \big(  (\tp{W} \frac{k}{4 \pi} T_0^{-1} W)_* \big)^{n+1+\epsilon}}
      \label{fcbound-prelim}\\
    & \ll_n \frac{ k^{(n+1)(\frac{5n}{4}-r)-nk+\epsilon} \expn{nk} (4 \pi)^{nk}
               \det(T_0)^{k-(n+1)/2+\epsilon}}
             {\det \big( (\tp{W} \det(T_0)^{-1} \Adj(T_0) W)_*\big)^{n+1+\epsilon}}
      \dn\\
    & \ll_n \expn{nk} (4\pi)^{nk} k^{(n+1)(\frac{5n}{4}-r)-nk}
    \det(T_0)^{k + (n+1)(r-1/2+\epsilon)} . \label{fcbound-badt}
  \end{align}
Here $\Adj(T)= \det(T) T^{-1}$ is the adjoint matrix to $T$ and the last inequality follows 
since \linebreak $\det \big((\tp{W} \Adj(T_0)W)_* \big) \ge 2^{-r}$, as $\Adj(T_0)$ is positive definite and has half-integer entries.

  When $r=n$ we obviously have a better exponent of
  $\det(T_0)$ as $W=1_n$ in this case.
  From \eqref{fcbound-prelim} directly, we have the exponent $\det(T_0)^{k+(n+1)/2+\epsilon}$.
  Otherwise, the worst case with respect to $\det(T_0)$ occurs when $r=n-1$.
  Here the final exponent can be bounded by $\det(T_0)^{k+(n+1)(n-3/2)+\epsilon}$.
  
  The worst exponent with respect to $k$ occurs when $r=0$. In this case,
  it is $k^{5n(n+1)/4 -nk +\epsilon}$. Comparing these with \eqref{fc1}, we see that
\begin{align*}
\big(\sum_{F\in\basis}|a_F(T_0)|^2 \big)^{1/2} \ll_n \frac{ ( 4\pi)^{nk/2} k^{5n(n+1)/8 - n/4 +\epsilon} |T_0|^{ k/2 + (n+1)(n/2-3/4) +\epsilon} }
  {\Gamma(k)^{n/2}   },
\end{align*} 
which gives the lemma.
\end{proof}

\begin{rmk} \label{overcount}
\begin{inparaenum}[(1)]
\item
The Resnikoff-Salda\~{n}a  conjecture (\cite{resnikoff}) predicts that $\beta= (n+1)/4$ in \eqref{fc0} with any implied constant depending only on $F$, but perhaps this is true with the given uniformity with respect to $k$ as well.

\item Using the results in this paper, we have the bounds $\bkzz \ll_n k^{(5n-3)(n+1)/4}$ everywhere in $\hn$. We can put this bound directly in \eqref{fbdbk} to obtain \eqref{fc1} with $(\alpha, \beta)=((5n^2-3)/8,0)$. We still need \lemref{beta0} to bootstrap.

\item
By this method, even the conjectural bound $\bkzz \ll k^{3n(n+1)/4 +\epsilon}$ would give rise to the pair $(\alpha, \beta)=(3n(n+1)/8 +\epsilon, 0)$.

\end{inparaenum}
\end{rmk}

For convenience of notation, we define (with $c_{n,k}$ as in \eqref{cnk})
\begin{equation} \label{pt}
  p(T)= p_{n,k}(T):=  \summ_{F \in \mathcal B^n_k} |a_F(T)|^2 = c_{n,k}^{-1} \det(T)^{k-\frac{n+1}{2}} a_{P_T}(T) .
\end{equation}

\begin{lem} \label{bkconnect}
For all  $Z \in \hn$ one has
\begin{equation}
\bkzz=\summ_{F \in \bnk}\det(Y)^k |F(Z)|^2 \le  \left( \summ_{T \in \Lambda_n} p(T)^{1/2} \, \det(Y)^{k/2} \, \exp(- 2 \pi \tr(TY)) \right)^2 .
\end{equation}
\end{lem}

\begin{proof}
The proof follows by the Cauchy-Schwartz inequality. Namely, we have
\begin{align*}
    \bkzz &\le \det(Y)^k \sum_{S,T} \exp(- 2 \pi \tr (S+T)Y) \sum_F |a_F(S)| \, |a_F(T)| \\
    & \le \det(Y)^k \sum_{S,T} \exp(- 2 \pi \tr (S+T)Y) \big( \sum_F |a_F(S)|^2 \big)^{1/2} \cdot \big( \sum_F |a_F(T)|^2 \big)^{1/2} \\
    & \le \det(Y)^k \big(  \sum_{T \in \Lambda_n} p(T)^{1/2} \exp(- 2 \pi \tr(TY)) \big)^2. \qedhere
\end{align*} 
\end{proof}

\begin{thm} \label{f-febd}
  For all $Z \in \fn$ with $Y= \Im(Z)$ and any $0 <\epsilon<1$,
  \begin{align}\label{eq:lem3statement}
    \sqrt{\bkzz}
    \ll_{n, \epsilon} 
    \left(\frac{k^{n}}{\det(Y)}\right)^{(n+1)/4 + \gamma_n}  \cdot k^{5n(n+1)/8+\epsilon} +
    \expn{- c_0 k^\epsilon}\det(Y)^{-\frac{n+1}{2}}
  \end{align}
 where $c_0>0$ depends only on $n$. Furthermore, if the largest diagonal entry $y_n>nk\,r_n/(2\pi)$ (with $r_n$ as in \eqref{reduct-implication})we have the exponential decay:
 \begin{align}
     \bkzz \ll_n \expn{-c_0 y_n}.
 \end{align}
\end{thm}
We immediately note that in the region $\{ Z \in \fn \mid \det(Y) \gg_n k^n \}$, the contribution from the Fourier expansion is at most $k^{5n(n+1)/8+\epsilon}$. By moving slightly higher in the cusp if necessary one gets exponential decay, see \lemref{expdecay-cutoff}.
\begin{proof}
In this proof we will use the Siegel fundamental domain $\fundamentalDomain$.
From \lemref{bkconnect}, we see that $\bkzz \le q_k(Y)^2 $,
where we have put
\begin{align} \label{qkdef}
  q_k(Y):=  \sum_{T \in \Lambda_n} p(T)^{1/2} \, \det(Y)^{k/2} \, \expn{- 2 \pi \tr(TY)} .
\end{align}
From \lemref{beta0}, we get that $p(T)^{1/2}$ satisfies \eqref{fc1} for $(\alpha, \beta)=((5n^2+3n)/8, -\gamma_n)$.
The proof now follows if we apply the bound \eqref{n/4bd} in \propref{p-febd} given below to $q_k$. We can use \eqref{n/4bd} since $Y\gg 1_n$ in $\fundamentalDomain$. The assertion about the exponential decay also follows from \eqref{expdbd} in \propref{p-febd}.
\end{proof}
From the above, it remains to analyse the growth of the function $q_k(Y)$ closely, as it will be the backbone of our Fourier series calculations. Such properties of $q_k$ are summarized in the following proposition. Apart from its general usage, it will be used in section~\ref{n2-fouriercoeffs} and subsection~\ref{lipsec1}.

\begin{prop} \label{p-febd}
Suppose that $p(T)=p_{n,k}(T)$ satisfies the bound \eqref{fc1}
\begin{equation}
  p(T)^{1/2} \ll_{n}
  \frac{ ( 4\pi)^{nk/2} k^{\alpha} \det(T)^{ k/2-\beta} }
  {\Gamma(k)^{n/2}}
.\end{equation}
Then for the function $q_k \colon \Sym{\R}^+ \to \complex$ defined in \eqref{qkdef}, we have
  \begin{equation} \label{n/2bd}
    q_k(Y)
    \ll_{n, \epsilon}
    \left(\frac{k^{n}}{\det(Y)}\right)^{(n+1)/2 -\beta}  \cdot k^{\frac{n}{4}+\alpha} +
    \expn{-c_0 k^\epsilon}\det(Y)^{-(n+1)/2},
  \end{equation}
for $Y$ reduced. Additionally, if we restrict to $Y\gg 1_n$, we have the better bound
  \begin{equation} \label{n/4bd}
    q_k(Y)
    \ll_{n, \epsilon}
    \left(\frac{k^{n}}{\det(Y)}\right)^{(n+1)/4 -\beta}  \cdot k^{\frac{n}{4}+\alpha+\epsilon} +
    \expn{-c_0 k^\epsilon}\det(Y)^{-(n+1)/2}
  .\end{equation}
  Here $c_0>0$ depends only on $n$ and $0<\epsilon<1$.
  Finally, if the largest diagonal entry $y_n>nk\,r_n/(2\pi)$ we have the decay:
  \begin{equation} \label{expdbd}
    q_k(Y) 
    \ll_{n} \expn{-c_0 y_n} \det(Y)^{-(n+1)/2}
  .\end{equation}
\end{prop}

\begin{rmk} \label{expected-fcbd}
\begin{inparaenum}

\item
If we estimate $p(T)$ using Poincar\'e series, we would get $\alpha= (3n^2+n)/8$. Suppose that we also get $\beta=0$. Then multiplying the bound for $\bkzz^3$ from \eqref{n/4bd} with the square of the second bound from \thmref{bkbd}, we get 
\begin{equation}
    \bkzz^5 \ll_n k^{3n(n+1)/2+ n(n+1) } \cdot k^{ 9n(n+1)/4+\epsilon } = k^{19n(n+1)/4+\epsilon},
\n\end{equation}
and thus $\bkzz \ll k^{19n(n+1)/20}$ for all $Z \in \fn$ and so also in $\hn$.

\item
Extending the above analysis for values of $\beta$ in the range $0\le\beta\le(n+1)/4$, we can show
\begin{equation}
    \bkzz \ll_n k^{\frac{n(n+1)}{4}\left(\frac{19(n+1)-40\beta}{5(n+1)-8\beta}\right)+\epsilon}
\end{equation}
At $\beta=(n+1)/4$, using \eqref{n/4bd} alone, we would get
the conjectural upper bound $k^{3n(n+1)/8+\epsilon}$ for $B_k(Z,Z)$ (cf. \eqref{bergconj}) up to the $\epsilon$.
\end{inparaenum}
\end{rmk}

\begin{proof}
  Taking \eqref{fc1} as a bound for $p(T)^{1/2}$, we see that
  \begin{equation} \label{eq:lem3expanded_condition}
   q_k(Y) \ll_n
    \sum_{T\in\Lambda_n}
            {\frac{(4\pi)^{\frac{nk}{2}}\det(TY)^{\frac{k}{2}}e^{-2\pi\tr(TY)}
                    k^{\alpha}\det(T)^{-\beta}
                  }
                  {\Gamma(k)^{\frac{n}{2}}}
            }
  .\end{equation}
  Let $Y^{1/2}$ be the unique symmetric positive definite square root of $Y$.
  The matrix $TY$ is conjugate to the positive definite symmetric matrix
  $Y^{1/2}TY^{1/2}$. So it is diagonalizable in $\R$ with
  positive eigenvalues. Let these eigenvalues be $\lambda_j, j\in\{1,...,n\}$.
  Replacing $\det(TY)$ with the product of these eigenvalues and $\tr(TY)$ with their sum
  and applying the Stirling's approximation
  $\Gamma(k)\asymp \sqrt{\frac{2\pi}{k}}(k/e)^k$, we get
  \begin{equation}
   q_k(Y) \ll_n
    \sum_{T\in\Lambda_*}{\left(
              \prod_j{\left(\sqrt{\frac{4\pi\lambda_j}{k}}\right)}^{k}
              \expn{\frac{nk}{2}-2\pi\tr(TY)}
                  k^{\alpha+\frac{n}{4}}\det(T)^{-\beta}
            \right)}
  .\dn\end{equation}
  We now look at a function $m$ defined as follows which captures the contribution of each summand in terms of the eigenvalues of $TY$:
  \begin{equation} \label{femass}
    q_k(Y) \ll_n
    \sum_{T\in\Lambda_*}{\left(
               \left(\prod_{j}{m(\lambda_j)}\right)
               k^{\alpha+\frac{n}{4}}\det(T)^{-\beta}
             \right)}; \q \q m(x) := \left(\frac{4\pi x}{k}\right)^{k/2} \expn{(k-4\pi x)/2}.
  \end{equation}
    
  This function $m$ attains a maximum value of $1$
  when $x$ is equal to $k/4\pi$.
  If $|\frac{x-k}{4\pi}| \gg_\epsilon k^{\frac{1}{2}+\epsilon}$, then
  $m(x) \ll \expn{\frac{-k^{2\epsilon}}{4}+O(k^{-1/2+3\epsilon})}$ for
  $0<\epsilon<1/2$.
  We claim that the total contribution from all $T$ for which any of the eigenvalues
  of $TY$ lie sufficiently away from $\frac{k}{4\pi}$ has exponential decay
  with respect to $k$. We will prove this claim later in \lemref{lem:expdecay}.
  So, the values of $T$ that make a significant contribution
  to the summation in \eqref{eq:lem3expanded_condition},
  are those for which the eigenvalues of $TY$ satisfy the following criterion
  \begin{equation}\label{eq:eigenvalue_restriction}
  \lambda_i \in \frac{k}{4\pi} + O(k^{\frac{1}{2}+\epsilon}) \q \text{ for all } 1 \le i \le n.
  \end{equation}

  Let us define the subset $\cy$ of $\Lambda_n$ by
  \begin{align} \label{cydef}
    \cy:=
    \left\{ T\in\Lambda_n \midmid \text{ all eigenvalues of } TY
    \text{ are of magnitude } \frac{k}{4 \pi} + O(k^{\frac{1}{2} +\epsilon})
    \right\}
  .\end{align}

  From the above discussion, we can write 
  \begin{equation} \label{eq:febd}
     q_k(Y)
    \ll_n \left( 
      \sum_{T \in \cy}{1} + \sum_{T \notin \cy}{\prod_j m(\lambda_j)}
    \right)            
    k^{\alpha+\frac{n}{4}}\det(T)^{-\beta}
  .\end{equation}  
  For the first summation above, i.e., when $T \in \cy $,
  we have $\det(TY) \asymp (k/(4\pi))^{n}$ from the eigenvalue condition.
  Moreover since $\det(T) \ge 2^{-n}$, $\cy$ is empty if $\det(Y)\gg_n k^n$.
  The contribution from the first sum is 
  \begin{align} \label{cysum}
    \sum_{T \in \cy}k^{\alpha+\frac{n}{4}}\det(T)^{-\beta}
    \ll_{n} \, \# \cy \cdot
    \det(Y)^{\beta} k^{-n \beta}
    k^{\alpha+\frac{n}{4} }
  .\end{align}

  Now we treat this as a counting problem on $T$.\label{count1}
  We provide two bounds, which are given in \lemref{count-bd} and \lemref{count-bd2}.
  For the first one we simply note that $\cy$ is contained in the set
  $\{ T \mid \tr(TY) \asymp \frac{nk}{4\pi} \}$ where the implied
  constant may depend only on $n$. And then we proceed using \lemref{count-bd} to get $\#\cy\ll_n k^{n(n+1)/2}\det(Y)^{-(n+1)/2}$ which gives \eqref{n/2bd}. If we take into account
 \lemref{lem:expdecay} below, we see that the contribution of the second summation in \eqref{eq:febd} is negligible (having a sub-exponential decay).
  
  If we are also given $Y\gg 1_n$, similarly we can use \lemref{count-bd2} to get \eqref{n/4bd}.
  
   Finally we use \lemref{expdecay-cutoff} when $\det(Y)$ is large getting exponential decay. This will complete the proof of the theorem.
\end{proof}

\begin{lem} \label{count-bd}
  For $Y\in R_n$ the Minkowski reduced domain, we can bound the size of the set
  \begin{equation}
    \#\{T\in\Lambda_n|\tr(TY)\le ck\} \ll_{n,c}
    k^{\frac{n(n+1)}{2}}\det(Y)^{-\frac{n+1}{2}}
  .\end{equation}
\end{lem}
\begin{proof}
  Since $Y$ is reduced, we can assume that it is almost equal to its diagonal $Y_D$.
  If the diagonal elements of $T$ are $\{ t_1, t_2, \cdots, t_n \}$, then we have 
  \begin{align}
    \tr(TY_D) = \sum_i t_i y_i \ll \tr(TY) \ll k
    \implies 1 \le t_i \ll k/y_i.
  \dn\end{align}
  Since $T>0$, when $i \neq j$, one has $|2 t_{i,j}| \le (t_i t_j)^{1/2}$, from which it follows that the number of choices for $T$ is, up to a constant depending only on $n$,
  \begin{equation*} 
  \ll_n
  \prod_i \frac{k}{y_i} \cdot \prod_{i<j} \frac{k}{(y_i y_j)^{1/2}} = \frac{k^{n+n(n-1)/2}}{\det(Y)^{1+(n-1)/2}} = \frac{k^{n(n+1)/2}}{\det(Y)^{(n+1)/2}}. \qedhere
  \end{equation*} 
\end{proof}

The following lemma gives a different bound for the quantity $\cy$. It generalises \cite[Lemma~4]{blomer2015size} to higher degrees -- uniform in $\det(Y)$.
\begin{lem} \label{count-bd2}
  Let $\cy$ be as in \eqref{cydef}, $Y$ is Minkowski reduced and $Y\gg 1_n$.
  Then we have the following bound on its size:
  \begin{align}
    \#\cy
    &\ll_n k^{\frac{n(n+1)}{4}+\epsilon}\det(Y)^{-\frac{n+1}{4}}. \label{new-count}
  \end{align}
\end{lem}
\begin{proof}
  Put $\kappa=O(k^{1/2+\epsilon})$ so that
  all eigenvalues of $Y^{1/2}TY^{1/2}$ are in the range $(\frac{k}{4\pi}-\kappa,\frac{k}{4\pi}+\kappa)$.
  
  Put $J:=(\frac{k}{4\pi}-\kappa)1_n$ and $K=(k_{i, j}):=Y^{1/2}TY^{1/2}-J$.
  $K$ is still positive definite.
  Its operator norm is bounded above by $2\kappa$ as well.
  By equivalence of operator norm and sup norm for positive definite matrices,
  there exist a constant $c_1$ only dependent on $n$ such that
  $\norm{K}_{\infty}\le 2c_1\kappa$. i.e., $|k_{i,j}|<2c_1\kappa$ for all $i, j$.
  We can write
  \begin{equation}
    T-Y^{-1/2}JY^{-1/2}=Y^{-1/2}KY^{-1/2}.
  \dn\end{equation}
  Let $(m_{i, j})$ be the entries of $Y^{-1/2}JY^{-1/2}$ and $(w_{i, j})$ be the entries of $Y^{-1/2}$.
  Then we calculate
  \begin{align}
    (Y^{-1/2}KY^{-1/2})_{i, j}
    =\sumn_{1\le p, q\le n}w_{i,p}k_{p,q}w_{q,j}
  .\dn\end{align}
  As $Y^{-1/2}$ is positive definite, $2|w_{i, j}|\le \sqrt{w_{i, i}w_{j, j}}$.
  Further as $Y$ is reduced in the Minkowski sense, from \eqref{reduct-implication},
  we have $r_n^{-1}Y<Y_D<r_nY$.
  Using the fact the if $A<B$ then $\sqrt{A}\le \sqrt{B}$ (cf. \cite{bapat1999linearalgebra}) we have
  $r_n^{-1/2}Y^{-1/2}\le Y_D^{-1/2}\le r_n^{1/2}Y^{-1/2}$ and we get
  $r_n^{-1/2}w_{i, i}\le y_i^{-1/2}\le r_n^{1/2}w_{i, i}$. This yields
  \begin{align}
    |(Y^{-1/2}KY^{-1/2})_{i, j}|
    <\sum_{1\le p, q\le n}2c_1r_n\kappa(y_iy_jy_py_q)^{-1/4}
  .\end{align}
  Recall that $Y\gg 1_n$ and use this to bound $y_p$ and $y_q$. From this we conclude
  $|t_{i, j}-m_{i, j}|\le2n^2c_1r_n\kappa(y_iy_j)^{-1/4}$.
  
  The number of possible values $t_{i, j}$
  can take is $\ll_n \kappa/\sqrt[4]{y_iy_j}+1$. Put back $\kappa=O(k^{1/2+\epsilon})$,
  and observe that if any $y_i\gg k$, then $\tr(TY)\gg k$
  and $\cy$ would be empty.
  Hence $k^{1/2+\epsilon}/\sqrt[4]{y_iy_j} \gg 1$.
  Multiplying the count for all $t_{i, j}$ gives us the total size of $\cy$ and proves the lemma.
\end{proof}

\begin{lem} \label{lem:expdecay}
  With $\cy, m$ and $\lambda_j$ as defined in \propref{p-febd}, we have for all $0 <\epsilon <1$, and for some absolute constant $c_0>0$ a `sub-exponential' type decay as follows.
  \begin{equation} \label{eq:expdecay}
  \sumn_{T \notin \cy}{\prod \nolimits_j m(\lambda_j)}
    k^{\alpha+\frac{n}{4}}\det(T)^{-\beta}
    \ll_n \exp \left( - c_0 k^{\epsilon} \right)\det(Y)^{-\frac{n+1}{2}}.
  \end{equation}
\end{lem}

\begin{proof}
  We will bound the summation by dividing the set $\Lambda_n\setminus\cy$ into regions $\cy^t, t\ge 0$.
  We will find the maximum value the summand can attain in each region and a bound on the number of elements in each region.
  The regions will be characterised by the largest eigenvalue $\lambda_n$ of $TY$.
  Define
  \begin{align}
  \cy^0 &:= \left\{ T\in\Lambda_n\setminus \cy
          \midmid 4\pi\lambda_{n}\le 2k
          \right\}\\
  \cy^t &:= \left\{ T\in\Lambda_n\setminus \cy
          \midmid 2^{t} < 4\pi\lambda_{n}\le 2^{t+1}k
          \right\}, \q t\geq 1 \dn
  .\end{align}
  If $T \in \cy^0$, then let $\lambda$ be an eigenvalue of $TY$ such that $|\lambda - \frac{k}{4 \pi}| \ge \frac{c}{4 \pi} k^{1/2+\epsilon}$ for a constant $c>0$ depending only on $\epsilon$. Recall that the function $m(x)$ has a single maximum at $x=\frac{k}{4 \pi}$ with a value equal to $1$.
  In either side, by monotonicity of $m$, we see that $ \prod_j m(\lambda_j) \ll_n m(\lambda)$. Now
 \begin{align*}
  m&\left(\frac{k\pm c k^{\frac{1}{2}+\epsilon}}{4\pi}\right) 
  =  \exp \left(\frac{k}{2}\log\left(1\pm c k^{-\frac{1}{2}+\epsilon}\right)
      \mp \frac{c}{2} k^{\frac{1}{2}+\epsilon}
    \right)\\ 
    &=  \exp \left( \frac{k}{2} \big( \pm c k^{-\frac{1}{2}+\epsilon} - c^2 \frac{k^{-1+2\epsilon}}{2}
      \pm O ( k^{-\frac{3}{2}+3\epsilon} ) \big)
      \mp \frac{c}{2} k^{\frac{1}{2}+\epsilon}\right) \\
    & = \exp \left(  - c^2 k^{2\epsilon}/ 4+O (k^{-\frac{1}{2}+ 3\epsilon}  ) \right)\ll \expn{- c^2 k^{2\epsilon} /4},
  \end{align*}
  Therefore, we get the following sub-exponential decay
  \begin{align}
  \prodd_j m(\lambda_j) \ll \expn{- c^2 k^{2\epsilon} /4}.
  \end{align}
  Further for any $t\geq 1$, we have $
    m\left(\frac{2^tk}{4\pi}\right) = \sqrt{2^{kt} e^{k(1-2^t)}}
    = e^{\frac{k}{2}(t\log(2)+1-2^t)} $.
  If $\beta<0$, in $\cy^t$, $\det(T)^{-\beta}\le (2^{t+1}k)^{-n\beta}$.
  Otherwise, we can simply use $\det(T)^{-\beta}\le 2^{n\beta}$.
  The number of points in these $\cy^t$ are bounded by (using \lemref{count-bd})
  \begin{equation}
    \#(\cy^t) \ll_n(2k)^{\frac{(t+1)n(n+1)}{2}}\det(Y)^{-\frac{n+1}{2}}
  .\dn\end{equation}
  We bound the summation by summing over count of points in each region $\cy^t$
  times a bound on the  maximum value $\prod_j m(\lambda_j)\det(T)^{-\beta}$ could take in that region.
  \begin{multline*} 
  \sumn_{T\in\Lambda_n\setminus\cy}{  \prod \nolimits_j {m(\lambda_j)}\det(T)^{-\beta}}
  \le \sumn_{t=0}^{\infty}{\left(\#(\cy^t)
    \max \nolimits_{T\in \cy^t}{\prod \nolimits_{j}{m(\lambda_j)}\det(T)^{-\beta}}
    \right)} \\
    \ll \biggl(\exp(- c_1 k^{2\epsilon})
         (2k)^{n(n+1)}\max\left((2k)^{-n\beta},2^{n\beta}\right)\\
         + \sumn_{t=1}^{\infty} \exp( k(\log(2)t+1-2^t) )(2k)^{ \frac{(t+1)n(n+1)}{2} }
         \max\left((2^{t+1}k)^{-n\beta}, 2^{n\beta}\right)\biggr)
         \det(Y)^{-\frac{n+1}{2}}
  \end{multline*}
It is easy to check that the above sum is bounded by $\exp(- c_0 k^{2 \epsilon} )\det(Y)^{-(n+1)/2}$. Replacing $\epsilon$ by $\epsilon/C$ for some $C>0$, if necessary, gives us the lemma.
\end{proof}

\begin{lem} \label{expdecay-cutoff}
  For the constant $C_n=\frac{2 \pi}{n \, r_n}$, in the region $\{\, Z \in R_n \mid y_n \ge k/C_n \,\}$, we have $q_k(Y) \ll_n \expn{-c_1 y_n}\det(Y)^{-(n+1)/2}$ for some constant $c_1>0$ depending only on $n$.
  
  Alternatively in the region $\{\, Y \in R_n \mid \det(Y) \ge (\frac{k}{C_n})^n \,\}$ the quantity $q_k(Y)$ decays exponentially: that is, $q_k(Y)\allowbreak \ll_n \exp(- c_1 \sqrt[n]{\det(Y)})$.
\end{lem}
We note that this also means exponential decay with respect to $k$ from the condition on $Y$.
\begin{proof}
  Let $\det(Y)=(bk/C_n)^n$ for some $b\ge 1$.
  In the above-mentioned region, by Hadamard's inequality we also have $y_1y_2 \cdots y_n \ge (\frac{bk}{C_n})^n $. Therefore, we must have $y_n \ge \frac{bk}{C_n}$. 

  As in the proof of \propref{p-febd}, if $\lambda_1 \le \lambda_2 \le \cdots \le \lambda_n$ denote the eigenvalues of $TY$, then it is easy to see $n \lambda_n \ge \tr(TY) \ge r_n^{-1} \sum_{j=1}^n t_i y_i \ge y_n/r_n$. Here we have used that $Y \ge Y_D/r_n$ from \eqref{reduct-implication}.
  Therefore, we get that $\lambda_n \ge \frac{bk}{n \, C_n \, r_n}$.

  If we choose $C_n$ such that $\frac{1}{n \, C_n \, r_n} \ge \frac{1}{2 \pi}$, i.e., $c_n \le \frac{2 \pi}{n \, r_n}$, then we must have $\lambda_n \ge \frac{bk}{2 \pi}$. This would mean that for all $T \in \Lambda_n$, $T \not \in \mc C_Y$. In fact,
  \begin{align}
    m(\lambda_n) \le m(bk/(2\pi)) =\expn{-k(2b-1-\log(2b))/2}  \le \expn{-c''bk}
  \dn\end{align}
  for some appropriate constant $c''$ and $k$ large enough. We proceed in the same way as that of \lemref{lem:expdecay} by considering the regions similar to $\cy^t$ ($t\ge 1$) characterised by $2^tbk<4\pi\lambda_n\le 2^{t+1}bk$ to get $q_k(Y)\ll_n\expn{-c_0bk}\det(Y)^{-(n+1)/2}$. We do not repeat the arguments again.
  We do note the absence of the $\cy^0$ region and hence we get full exponential decay
  (as opposed to a sub-exponential decay).
  Put back $b=C_n\sqrt[n]{\det(Y)}/k$ to finish the proof.
\end{proof}

\section{Using the Petersson trace formula when \texorpdfstring{{$n=2$} }{n2}} \label{n2-fouriercoeffs}
In this section, we will get hold of a better bound on the Fourier coefficients, when $n=2$. Since in degree $2$ there is sufficient information about Bessel functions and Kloosterman sums; we bound $a_F(T)$ simply by bounding the $T$-th Fourier coefficient $a_{P_T}(T)$ of the Siegel Poincar\'e series $P_T$ and use the Petersson trace formula \eqref{peter}.

We recall the definition of the Poincar\'e series $P_Q(Z)$ ($Q \in \Lambda_n$):
\begin{equation}
  P_Q(Z):=\sum_{\gamma \in \Gamma_{\infty}\backslash\Gamma_n}
           {\det(j(\gamma, Z))^{-k}e^{2\pi i\tr(Q\gamma(Z))}}.
\end{equation}
Here $\Gamma_{\infty}$ is the subgroup of $\Sp{\z}$ with elements of the form
$\smat{I}{B}{0}{I}$.
Work towards bounding its Fourier coefficients $a_{P_Q}(T)$ has been done in
\cite{kitaoka1984fourier} when $n=2$. It shows
\begin{equation}
  |a_{P_Q}(T)| \ll_{k,Q} \det(T)^{k/2-1/4+\epsilon}.
\dn\end{equation}
However, this result does not provide an explicit dependence on the weight $k$
for these Fourier coefficients, which is crucial for us.
The dependence on $k$ has been tracked in \cite[Proposition 3.3]{kowalski2012local} and \cite[Theorem 5.11]{dickson2015thesis} but
only for certain special type of matrices $Q$ and $T$. Also, $Q$ was treated as fixed.
We are interested in the case when $Q=T$, with $T$ varying.
We follow the approach in \cite[p.355-363]{kowalski2012local} with this additional bookkeeping.
\begin{lem}
The $T$-th Fourier coefficient of the $T$-th Poincar\'e series can be bounded as follows
\begin{equation} \label{eq:poincare_fourier_bound}
    a_{P_T}(T) = \delta(T, T) + O_{\epsilon}(k^{-2/3}\det(T)^{1+\epsilon}).
\end{equation}
Here $\delta(T, Q)=\#\{U\in\GL[2]{\z}\mid UQ\tp{U}=T\}$.
\end{lem}
\begin{proof}
From \cite[p.~158]{kitaoka1984fourier} one can write
\begin{equation}
  P_Q(Z)=
  \sumn_{M\in\Gamma_{\infty}\backslash\Gamma_n/\Gamma_{\infty}\q}
  \sumn_{T\in\Lambda_n}
  h_Q(M, T)e^{2\pi i\tr(TZ)},
\dn\end{equation}
where 
\begin{equation}
  a_{P_Q}(T)=
  \sumn_{M\in\Gamma_{\infty}\backslash\Gamma_n/\Gamma_{\infty}}
  h_Q(M, T)
.\dn\end{equation}
Note that we deliberately write $h_Q$ instead of $h_T$ to be consistent with the notation in \cite{kitaoka1984fourier}, \cite{kowalski2012local}. We reiterate, for us $Q=T$. In the rest of this section, we will closely follow these two references and indicate only the appropriate changes required to achieve our goal.

Then if $M$ is expressed as $\smat{A}{B}{C}{D}$, the above summation
is divided into cases depending on the rank of $C$. When $\rk C=0$,
$h_Q(M, T) = \delta(Q, T)$ and so their sum over all $M$ with $C=0$
will be the size of the orthogonality group $\#\mrm{Aut}(T)=O(1)$.

When $\rk C=1$, it has been shown in \cite[(3.1.8)]{kowalski2012local} that
\begin{equation}
  \sum_{\rk C=1} h_Q(m, T)
  \ll \sum_{c, m\geq 1} \left(\frac{\det(T)}{\det(Q)}\right)^{k/2-3/4}
  m^{-1/2}(c, m)^{1/2}
  \left|J_{k-3/2}\left(\frac{\pi\sqrt{\det(TQ)}}{mc}\right)\right|
\dn\end{equation}
\begin{equation} \label{rk1}
  \sum_{\rk C=1} h_T(m, T)
  \ll \sum_{c, m\geq 1}
  m^{-1/2}(c, m)^{1/2}
  \left|J_{k-3/2}\left(\frac{\pi\det(T)}{mc}\right)\right|
.\end{equation}
Here $J_{k}$ is the Bessel function.
Using various approximations for the Bessel function, it was shown that
the the r.h.s. of \eqref{rk1} is $\ll k^{-5/6}$ but with implicit constants dependent on $\det(T)$.
In all cases, $d$ appears in the aforementioned calculation and for our situation one has to replace $d$ with $\det(T)$.
This allows us to easily track the dependence to be $\det(T)^{1+\delta}$ for any $\delta>0$. Hence
\begin{equation}
  \sumn_{\rk C=1} h_T(m, T)
  \ll \det(T)^{1+\delta} k^{-\frac{5}{6}}
.\end{equation}
Finally, when $\rk C =2$, from \cite[(3.1.11)]{kowalski2012local} for a fixed value of $C$, 
\begin{equation}
  \sumn_{D (\mod C)} h_Q(M, T)
  \ll\left(\frac{|T|}{|Q|}\right)^{\frac{k}{2}-\frac{3}{4}}
  \lvert C\rvert^{-\frac{3}{2}}
  K(Q, T; C)
  \mathcal{J}_k(P(C))
\dn\end{equation}
\begin{equation}
  \sumn_{D (\mod C)} h_T(M, T)
  \ll \lvert C\rvert^{-\frac{3}{2}}
  K(T, T; C)
  \int_0^1 J_{k-\frac{3}{2}}(4\pi s_1 t) J_{k-\frac{3}{2}}(4\pi s_2 t)
  \frac{tdt}{\sqrt{1-t^2}}
.\dn\end{equation}
Here $K(Q, T; C)$ is a Kloosterman type sum and $(s_1^2, s_2^2)$ are the
eigenvalues of $T{\tp{C}}^{-1}TC^{-1}$. We refer the reader to \cite[p. 165]{kitaoka1984fourier} for the definition of $\mathcal{J}_k(P(C))$.
First from \cite[\S1 Proposition 1]{kitaoka1984fourier} we have
$K(T, T; C)\ll c_1^2c_2^{1/2}(c_2, t_{22})^{1/2}$ where we have written $C=U^{-1}\smat{c_1}{}{}{c_2}V^{-1}$ for
$U, V\in\GL{\z}$ and $1\le c_1|c_2$ and $t_{22}$ is the $(2, 2)^{\text{th}}$ entry of $T$.
Here again, in order to handle the Bessel functions, the summation was divided into different regions, which are
$\mathscr C_1, \mathscr C_2$ and $\mathscr C_3$ (cf. \cite[p. 166]{kitaoka1984fourier}),
depending on the sizes of $4\pi s_j$.

In $\mathscr C_1$, the Bessel integral is bounded by $(s_1s_2)^{2+\delta}2^{-k}$ and $\#(\mathscr C_1)\ll (s_1s_2)^{-1}$.
This gives us
\begin{equation}
  \sumn_{C\in\mathscr C_1} h_T(m, T)\ll \det(T)^{1+\delta}2^{-k}.
\dn\end{equation} 
Similarly, in $\mathscr C_2$, we can deduce from what was shown in \cite{kowalski2012local} that
\begin{equation}
  \sumn_{C\in\mathscr C_2} h_T(m, T) \ll \det(T)^{1/2+\delta}k^{-1}.
\dn\end{equation}
Finally, in $\mathscr C_3$, the Bessel integral is bounded by $k^{-2/3}$. Note that the sum over $C$ is infinite, and using the elementary divisors we can write
\begin{align}
  \sumn_{C\in\mathscr C_3} h_T(m, T)
  \ll k^{-2/3}\sumn_{C\in\mathscr C_3}{c_1^{1/2}c_2^{-1}t_{22}^{1/2}}
  \ll k^{-2/3}\det(T)^{1/2}
.\end{align}
Combining all these cases concludes the proof.
\end{proof}
Let $\basis[2]$ be an orthonormal basis for $M_{K}(\Sp{\z})$.
The Poincar\'e series also satisfy the Petersson trace formula
\begin{equation}
   a_{P_T}(Q) = \frac{(4\pi)^{\frac{7}{2}-2k}\Gamma(k-\frac{3}{2})\Gamma(k-2)}
                     {\det(T)^{k-3/2}}
                \sumn_{F\in\basis[2]}{a_F(Q)\overline{a_F(T)}}
.\dn\end{equation}
Now using this and \eqref{eq:poincare_fourier_bound}, we get
\begin{align} 
  p(T) =\sumn_{F\in\basis[2]}\left|a_F(T)\right|^2
  \ll \frac {(4\pi)^{2k} \det(T)^{k}} {\Gamma(k-2)\Gamma(k-3/2)}
  \left(\det(T)^{-\frac{3}{2}}+k^{-\frac{2}{3}}\det(T)^{-\frac{1}{2}+\epsilon}\right) \dn\\
  \sqrt{p(T)} \ll \frac
  {(4\pi)^{k} \det(T)^{\frac{k}{2}}}
  {\sqrt{\Gamma(k-2)\Gamma(k-3/2)}}
  \left(\det(T)^{-\frac{3}{4}}+k^{-\frac{1}{3}}\det(T)^{-\frac{1}{4}+\epsilon}\right) \label{n2-fourier-bound}
.\end{align}
This gives $\alpha=7/4$ and $\beta=1/4$ for
\thmref{f-febd} uniformly. This results in the bound
$\oldnorm{F}_{\infty} \ll k^{\frac{7}{2}+\epsilon}$.
However, we will keep note of the two terms in \eqref{n2-fourier-bound} separately. This observation will be crucial
in a more refined argument later in the next section.

\section{The \texorpdfstring{$n=2$}{n2} case} \label{n=2case}
We now put together all we know when $n=2$. \thmref{f-febd} relies
on counting a certain subset $\cy$ of $\Lambda_n$. This counting has been carried out more precisely in
\cite[Lemma 4]{blomer2015size} when $n=2$ by explicitly computing the eigenvalues of elements of $\cy$. We need an even more refined version of the calculation.
The results in this section establish a bound for growth/decay of $\log_k(\bkzz)$ with respect to $\log_k(\det(Y))$. It is easier to linearize the calculations by working with logarithms. As can be seen from below, most of our resulting expressions are piece-wise linear polynomials.

\begin{lem} \label{l-n2counting}
  Let $\cy$ be as defined in \eqref{cydef}. If $\eta:=\log_k(\det(Y))$
  and $n=2$, then $\#\cy\ll k^{w_1(\eta)}$ where
  \begin{equation}
    w_1(\eta) =
    \epsilon+
    \begin{cases}
    \frac{3}{2}-\frac{3\eta}{2} & (0\leq\eta\leq\frac{1}{2})\\
    1-\frac{\eta}{2}            & (\frac{1}{2}\leq\eta\leq 1)\\
    \frac{3}{2}-\eta            & (1\leq\eta\leq\frac{3}{2})\\
    0                           & (\frac{3}{2}\leq\eta\leq 2)
    \end{cases}
  .\end{equation}
\end{lem}
\begin{proof}
  According to the final equation in the proof of \cite[Lemma 4]{blomer2015size}
  \begin{equation} \label{blomers-count}
    \#\cy\ll k^{\epsilon}
    \left(
      \frac{k^{3/2}}{\det(Y)^{3/2}}
      +\frac{k}{y_1^{3/2}y_2^{1/2}}
      +\frac{k^{1/2}}{y_1}
      +1
    \right)
  .\end{equation}
  Here $y_1\leq y_2$ are the diagonal entries of $Y$. For a Minkowski reduced matrix, $\det(Y)\asymp y_1y_2$. They are individually
  bounded above by $O(k)$ (because of \lemref{expdecay-cutoff}). So we can bound $y_1\gg 1$ when $\eta<1$ and, $y_1\gg k^{\eta-1}$ for $1\le \eta \le 2$.
  Using this, we can write
  \begin{equation}
    \#\cy\ll k^{\epsilon}\times
    \begin{cases}
    \left(
      k^{\frac{3}{2}(1-\eta)}
      +k^{1-\frac{\eta}{2}}
      +k^{\frac{1}{2}}
      +1
    \right) & \eta\leq 1\\
    \left(
      k^{\frac{3}{2}(1-\eta)}
      +k^{2-\frac{3\eta}{2}}
      +k^{\frac{3}{2}-\eta}
      +1
    \right) & 1\leq\eta\leq 2
    \end{cases}
  .\end{equation}
  Now, checking which of the four terms is largest for a value of $\eta$ results in the lemma statement.
\end{proof}

\begin{lem} \label{l-n2fourierbound}
  Let $\eta:=\log_k(\det(Y))$. We have
  $\sqrt{\bkzz}\ll k^{w_2(\eta)}$ where
  \begin{equation}
    w_2(\eta) =
    \epsilon+
    \begin{cases}
    \frac{35}{12}-\frac{5\eta}{4} & (0\leq\eta\leq\frac{1}{2})\\
    \frac{29}{12}-\frac{\eta}{4}  & (\frac{1}{2}\leq\eta\leq 1)\\
    \frac{35}{12}-\frac{3\eta}{4} & (1\leq\eta\leq\frac{4}{3})\\
    \frac{9}{4}-\frac{\eta}{4}    & (\frac{4}{3}\leq\eta\leq\frac{3}{2})\\
    \frac{3}{4}+\frac{3\eta}{4}   & (\frac{3}{2}\leq\eta\leq 2)
    \end{cases}
  .\end{equation}
\end{lem}
\begin{proof}
  We utilise the fact that we have bounds on the Fourier coefficients of the form
  \begin{equation} \label{doubleab}
    p(T)\ll \frac{(4\pi)^k\det(T)^{k/2}}{\Gamma(k)}
    \left(k^{\alpha_1}\det(T)^{-\beta_1} + k^{\alpha_2}\det(T)^{-\beta_2}\right)
  .\end{equation}
  We define $p_1(T), p_2(T)$ such that $p_j(T)$ satisfies \eqref{fc1} for $(\alpha_j, \beta_j)$.
  We also define $q_k, q_{k, 1}$ and $q_{k, 2}$ from \eqref{qkdef} using $p(T), p_1(T)$ and $p_2(T)$ respectively.
  Then clearly $|q_k| < |q_{k, 1}| + |q_{k, 2}|$.
  Now we can apply \propref{p-febd} on the two $q_{k, j}$, but holding back the counting argument (cf. \eqref{cysum}), to arrive at
  \begin{align}
    q_k(Y)
    &\ll \#\cy k^{\frac{1}{2}}
    \left(k^{\alpha_1-2\beta_1}\det(Y)^{\beta_1}
         +k^{\alpha_2-2\beta_2}\det(Y)^{\beta_2}
    \right)\n\\
    & = \#\cy k^{\frac{1}{2}}
    \left(k^{\alpha_1-2\beta_1+\beta_1\eta}
         +k^{\alpha_2-2\beta_2+\beta_2\eta}
    \right)
  .\end{align}
  From \lemref{bkconnect}, $\sqrt{\bkzz}$ is bounded by $q_k(Y)$.
  From \eqref{n2-fourier-bound}, the expression \eqref{doubleab} is satisfied for
  $(\alpha_1, \beta_1)=(7/4, 3/4)$ and $(\alpha_2, \beta_2)=(17/12, 1/4-\epsilon)$.
  Thus we have
  \begin{equation}
    \sqrt{\bkzz}
    \ll \#\cy k^{\frac{1}{2}}
    \left(k^{\frac{1}{4}+\frac{3\eta}{4}}
         +k^{\frac{11}{12}+\frac{\eta}{4}+\epsilon }
    \right)
  .\end{equation}
  The second term is worse when $\eta\leq 4/3$ and otherwise, the first term
  is worse. We expand out the bound on $\#\cy$ from \lemref{l-n2counting}.
  These together give us the five cases in the lemma statement.
\end{proof}

\begin{thm} \label{9/4}
  For $n=2$, we have
  \begin{equation}
    \bkzz \ll_{\epsilon} k^{\frac{9}{2}+\epsilon}
  \end{equation}
\end{thm}
\begin{proof}
  From \thmref{bkbd}, we have
  \begin{equation}
    \bkzz
    \ll k^{3}\det(Y)^{\frac{9}{4}+\epsilon}
  .\end{equation}
  When $\det(Y)\leq k^{2/3}$, this bound shows
  $\bkzz\ll k^{9/2+\epsilon}$. When $k^{2/3}\leq\det(Y)\leq k^2$,
  from \lemref{l-n2fourierbound}, we get that $\sqrt{\bkzz}\ll k^{9/4+\epsilon}$.
  Finally when $\det(Y)\geq k^2$, \lemref{expdecay-cutoff}
  shows that there is exponential
  decay in the value of $\sqrt{\bkzz}$ with respect to $k$.
  This concludes the proof.
\end{proof}

\section{Bounds using the Bergman Kernel} \label{bksec1}
  From \cite[eqn. 6]{cogdell2011bergman},
  for an orthonormal basis $\bnk=\{F_{k,j} \}_j$ for Siegel Cusp forms of weight $k$, we have for $Z \in \hn$
  \begin{equation}
    \bkzz
    = \sum_{j}{|F_{k, j}(Z)|^2\det(Y)^k}
    = B_k(Z,Z)|\det(Y)|^k
    = 2^{-1}\ank R_k(Z),
  \end{equation} 
where we have put
 \[ R_k(Z):=    \sum_{\gamma\in\Gamma_n}{h_\gamma(Z)^k}; \q
 h_\gamma(Z)
      = \frac{\det(Y)}
         {\det\left(\frac{Z-\gamma \overline{Z}}{2i}\right)
          \det(C\overline{Z}+D)
         }
  \]
  and $\ank$ defined below satisfies
  \begin{equation} \label{ank}
    \ank
    := 2^{-n(n+3)/2}\pi^{-n(n+1)/2}
      \prod_{v=1}^{n}{\frac{\Gamma(k-\frac{v-1}{2})}{\Gamma(k-\frac{v+n}{2})}}
      \ll_n k^{\frac{n(n+1)}{2}}.
  \end{equation} 
  
  Let us write $\gamma\in\Gamma_n$ as
  $\gamma=\gamma_\infty\gamma_U\widetilde\gamma$ where
  \[\gamma_\infty=\begin{pmatrix}1_n & T\\ 0 & 1_n\end{pmatrix};
  \gamma_U=\begin{pmatrix}U & 0\\ 0 & (\tp{U})^{-1}\end{pmatrix};
  T\in \Sym \z, U\in \GL \z
  .\]
  For both $\gamma_\infty$ and $\gamma_U$, we have (with $j(\gamma,Z)=CZ+D$)
  \[ \det(j(\gamma_\infty, Z))=\det(j(\gamma_U, Z))=1
  .\]
  Thus we have
  \[ h_\gamma(Z)
       = \frac{\det(Y)}
         {\det\left(\frac{Z-U \overline{ \widetilde\gamma \lan Z \ran }\tp{U}-T}{2i}\right)
          \det(j(\widetilde\gamma, Z))
         }
  \]
  and so
  \begin{equation} \label{eq:h_gamma_breakdown}
    \sum_{\gamma\in\Gamma_n}h_\gamma(Z)^k
    =(2i)^{nk}\det(Y)^k\sum_{\{ C, D \}} \det(CZ+D)^{-k}
      \sum_{U} \sum_{T}
      \det\left(Z-U \overline{ \widetilde\gamma \lan Z \ran} \tp{U}-T\right)^{-k};
  \end{equation}
 where in the above sum, $T$ varies over $\Sym \z$, $U$ varies over $\GL \z$ and $\{C,D\}$ vary over inequivalent coprime symmetric pairs in degree $n$ (see section~\ref{prelim}).
 Next, we shall introduce a couple of lemmas to help us estimate this nested summation.

\subsection{Bounding \texorpdfstring{$\bkzz$}{bkz} using Lipschitz formula} \label{lipsec1}

\begin{lem} \label{lipz}
Let $k \ge n+1$ and $Z \in \hn$ be such that $Y \ge c1_n$ for some $c$ depending only on $n$. Then 
  \begin{equation} \label{lipf}
    \summ_{S\in \Sym \z} |\det(Z+S)|^{-k}
  \ll_n \det(Y)^{-k+(n+1)/2}.
  \end{equation}
\end{lem}

Note the subtle point here, that it is not sufficient to estimate the l.h.s. of \eqref{lipf} term by term as it is. This would entail bad bounds in terms of the exponent $k$ (cf. \cite[p.~393-395]{braun1939konvergenz}, also see \cite[Lemma~V.1.2, Corollary~V.1.5]{krieg1986modular}, where the relevant constants are exponential in $k$), but we can avoid this by separating the exponent. We will not use this lemma in this subsection but in section~\ref{n=1exam} (cf.~\eqref{bkzn1bd}). The argument given below using the comparison of higher dimensional sum and integral may be useful in other situations.

\begin{proof}
Let us write $k=k_1+k_2$ where $k_1$ will depend only on $n$. First, we note that for any $Z \in \hn$ and $R \in \Sym \R$, $|\det(Z+R)| \ge \det(Y)$, where $Y=\Im(Z)$. This follows from the computation
\begin{equation*}
|\det(Z+R)| = |\det(iY+(X+R))| \geq \det(Y) |\det(1_n - iY^{-1/2}(X+R)Y^{-1/2})| .
\end{equation*} 
From the above discussion, we can write
\begin{align} \label{split1}
\sumn_{S\in \Sym \z}|\det(Z+S)|^{-k} \le \det(Y)^{-k_2} \sumn_{S\in \Sym \z} |\det(Z+S)|^{-k_1} .
\end{align}
Quoting \cite[p.~393-395]{braun1939konvergenz} we have
\begin{align} \label{sumint}
\sumn_{S\in \Sym \z}|\det(Z+S)|^{-k_1} &\le \expn{n k_1 \tr(Y^{-1})} \int_{\mc S = \mc S'}| \det(Z+ \mc S)|^{-k_1} d \mc  S \\
& \ll_{n,k_1} \det(Y)^{(n+1)/2 - k_1},
\end{align}
provided $k_1 \ge n+1$ and that $Y \ge c 1_n$ for some $c$ depending only on $n$. Combining \eqref{split1} and \eqref{sumint} we get upon choosing $k_1=n+1$ that
\begin{align} \label{splitbd}
\sumn_{S\in \Sym \z}|\det(Z+S)|^{-k} \ll_n \det(Y)^{-k+k_1} \cdot \det(Y)^{(n+1)/2 - k_1}=\det(Y)^{-k+(n+1)/2}.
\end{align}
This finishes the proof of the 1emma.
\end{proof}

We now give a couple of other bounds for the sum in \lemref{lipz} using the Fourier expansion of the function $\sum_{S\in \Sym \z}\det(Z+S)^{-k}$.

\begin{lem} \label{lipfo}
For $Z \in \hn$ we have the bounds
\begin{align}
  \sumn_{S\in \Sym \z}\det(Z+S)^{-k} \ll_n
  \begin{cases}
    k^{\frac{n(n+1)}{4}} \det(Y)^{-k} \q \q &(Z \in \hn),\\
    k^{\epsilon}\det(Y)^{-k +\frac{n+1}{4}} \q \q &(Z \in \fn).
  \end{cases}
\end{align}
\end{lem}

\begin{proof}
First of all we can assume that $Y=\Im(Z)$ is reduced, by replacing $Z$ by $Z[U]$ for a suitable $U \in \GL \z$, since the sum over $S$ is invariant under $\GL \z$.

By the Lipschitz formula, which we quote from \cite[Hilfssatz~38]{siegel1935analytic}: for any $Z \in \hn$ and any $k > n(n+1)/2$ 
\begin{equation} \label{lips}
\mc S_{k}:=\sumn_{S \in \Sym \z} \det(Z+S)^{-k} = C_n \frac{e^{- \pi i n k/2}}{\gamma_{n,k}} \sum \nolimits_{T\in\Lambda_n} \det(T)^{k- (n+1)/2} e(TZ),
\end{equation}
where $C_n=(2 \sqrt{\pi})^{-n(n-1)/2}, \gamma_{n,k}= (2 \pi)^{-nk} \prod_{\nu =0}^{n-1} \Gamma(k-\nu/2)$. Asymptotically, we get the following
\begin{align}
\abs{C_n \frac{e^{- \pi i n k/2}}{\gamma_{n,k}}}
\asymp_n \frac{(2\pi)^{nk}k^{\frac{n(n-1)}{4}}}{\Gamma(k)^n}
\asymp_n \frac{ (4\pi)^{nl/2} l^{\frac{-n(n+2)}{4}} }{\Gamma(l)^{n/2}}.
\dn\end{align}
Here we have put $l=2k-n-1$. Thus we have
\begin{equation}
  \abs{\mc S_k \det(Y)^{l/2}}\asymp_n
  \sum_{T\in\Lambda_n} \frac{(4\pi)^{nl/2} l^{\frac{-n(n+2)}{4}} \det(TY)^{l/2} e^{2\pi\tr(TY)}}
  {\Gamma(l)^{n/2}}.
\end{equation}
We now apply \propref{p-febd} to the above with equation with $\alpha=\frac{-n(n+2)}{4}$,
$\beta=0$ and $k$ there being $l$ here. With the notation of \eqref{n/2bd}, the r.h.s of the above display is just $q_l(Y)$. Thus we have
\begin{align}
   \mc S_k \det(Y)^{l/2} \ll_{n,\epsilon}
   \det(Y)^{-(n+1)/2}(k^{\frac{n(n+1)}{4}}+\expn{-c_0k^{-\epsilon}})
\end{align}
which implies that
$\mc S_k \ll_{n} k^{\frac{n(n+1)}{4}}\det(Y)^{-k}$.

If we also have $Y\gg 1_n$, then we can also use \eqref{n/4bd}. This gives us
\begin{align*}
   \mc S_k \det(Y)^{l/2} \ll_{n,\epsilon}
   \det(Y)^{-(n+1)/4}k^{\epsilon}+\det(Y)^{-(n+1)/2}\expn{-c_0k^{-\epsilon}}
\end{align*}    
 which implies that $\mc S_k \ll_{n, \epsilon} k^{\epsilon}\det(Y)^{-k+\frac{n+1}{4}}$.
\end{proof}
This yields the following result.

\begin{lem} \label{lipcor}
Let $Z \in \hn$ and $k \ge n+1$. For $\gamma \in \Gamma_n$, we put $\widetilde Y:=-\Im( \overline{ \gamma \lan Z \ran } )$. Put $\mbb Y_U:= Y+U\widetilde Y\tp{U}$. Then the sum $\sum_{S\in \Sym \z}
    \det\left(Z-U \overline{ \widetilde\gamma \lan Z \ran } \tp{U}-S\right)^{-k}$ is
\begin{align} 
    \ll_{n, \epsilon}
    \begin{cases}
      k^{n(n+1)/4} \det(\mbb Y_U)^{-k} \q \q &(Z \in \hn) \\
      k^{\epsilon}\det(\mbb Y_U)^{-k +\frac{n+1}{4}} \q \q &(Z \in \fn).
    \end{cases}
\end{align}
\end{lem}

\begin{proof}
    We apply \lemref{lipfo} to $W:=Z-U \overline{ \widetilde\gamma \lan Z \ran } \tp{U}$. Notice that 
$\Im(W) = Y+ U \widetilde Y \tp{U} \ge Y$. Since $Z \in \fundamentalDomain$, $Y \gg_n 1$, and thus the second bound of the lemma follows. The first bound is valid for all $Z\in\hn$.
\end{proof}

\begin{lem} \label{headache}
  For $k > (n+1)^2, n \ge 2$ and $Y, \widetilde Y \in \Sym \R^+$, with $\ymin$ as the smallest eigenvalue of $Y$, we have for any $\epsilon>0$:
  \begin{equation}
  \summ_{U\in \GL{\z}}
      {\det(Y+U\widetilde Y\tp{U})^{-k}}
  \ll_{n,\epsilon}  2^{-nk}\det(Y)^{-\frac{k}{2}}\det(\ty)^{\frac{-k+n+1}{2}+\epsilon}\ymin^{-\frac{n(n+1)}{2}-n\epsilon}.
  \end{equation}
\end{lem}

Before we start the proof, we would like to call attention to the quantity $2^{nk}$ sitting in the front of the summations in \eqref{eq:h_gamma_breakdown}. We must be able to cancel this exponential factor -- as our quest is to obtain polynomial bounds on the sup-norm. The proof given below proceeds by taking this into account.

\begin{proof}
  We first notice that the roles of $Y$ and $\ty$ can be swapped without affecting the sum. Further, if $\ty_1$ satisfying $V\ty_1\tp{V}=\ty$ is the Minkowski reduction of $\ty$ for some $V\in\GL{\z}$ we can write
  \begin{equation}
    \sumn_U\det(Y+U\ty \tp{U})^{-k}=\sum_U\det((UV)^{-1} Y \tp{((UV)^{-1})} + \ty_1)^{-k}
    =\sumn_U\det(\ty_1+\tp{U} Y U)^{-k}
  .\dn\end{equation}
  We start with the obvious equality for any $\kappa >0$
  \begin{align} \label{obeq}
    \det(\ty_1 +\tp{U} Y U)^{-\kappa}
    = \det(Y)^{-\kappa} \det(1_n+ \ty_1^{-1/2} \tp{U} Y U \ty_1^{-1/2})^{-\kappa}.
  \end{align}
  
  We now put $A=\ty_1^{-1/2} \tp{U}  Y U \ty_1^{-1/2}>0$.
  Observe that for $A>0$ the following inequality holds:
  \begin{equation}
    \det(1_n +A) \ge 2^n \det(A)^{1/2}.
  \dn\end{equation}
  And thus from \eqref{obeq}, we have
  \begin{equation} \label{life}
    \det(\ty_1+\tp{U} YU)^{-\kappa} \le 2^{-n \kappa} \det(\ty_1)^{-\kappa}\det(A)^{-\kappa/2}
    =2^{-n \kappa} \det(\ty_1)^{-\kappa/2}\det(Y)^{-\kappa/2}.
  \end{equation}

  Also, using the inequality $\det(1_n+A)\geq 1 +\tr(A)+\det(A)$ for a positive semi-definite matrix $A$ (which follows easily by diagonalising $A$), we get
  $ \det(1_n+ A) \geq  (1+ \tr(\ty_1^{-1} \tp{U}  Y U)) + \det(A).$
  Then applying the A.M.-G.M. inequality, we get
  \begin{align} \label{tr-ineq}
    \det(1_n+ A) \geq 2 \left( 1+ \tr( \ty_1^{-1} \tp{U} Y  U) \right)^{1/2} \det(A)^{1/2}.
  \end{align}
  From \eqref{obeq} and \eqref{tr-ineq} we get
  \begin{align} \label{k'ineq}
    \det(\ty_1 +\tp{U} Y U)^{-\kappa} \le 2^{-\kappa} \det(\ty_1)^{-\kappa/2}\det(Y)^{-\kappa/2} \left( 1+ \tr( \ty_1^{-1} U Y \tp{U}) \right)^{-\kappa/2}.
  \end{align}

Like in the proof of \lemref{lipz}, we write $k=k''+k'$ with $k'$ depending only on $n$, to be specified later. This will also tell us how large $k$ needs to be. We use \eqref{life} for the exponent $k''$ and \eqref{k'ineq} for $k'$ to get that: 
  \begin{align} \label{usplit}
    \summ_U &\det(\ty +\tp{U}YU)^{-k}  = \summ_U \det(\ty +\tp{U} YU)^{-k'} \det(\ty +\tp{U} YU)^{-k''} \dn\\
    &\le 2^{-nk+(n-1)k'} \det(\ty)^{-k/2} \det(Y)^{-k/2} \summ_U \left( 1+ \tr( \ty_1^{-1} U Y \tp{U}) \right)^{-k'/2}.
  \end{align}

  For the rest of the proof, we will focus on the quantity $\mbb S:= \summ_U ( 1+ \tr( \ty_1^{-1} U Y \tp{U}) )^{-k'/2}$.
  As $\ty_1$ is reduced, we have $\ty_1\gg_n\tyd$, where $\tyd=\dia(\mu_1, \ldots,\mu_n)$ is the matrix made by the diagonal of $\ty_1$.
  Since $k'$ depends only on $n$, we can thus bound $\mbb S$ by
  \begin{align} \label{usum1}
  \mbb S \ll_n \summ_U \left( 1+ \tr( \tyd^{-1} \tp{U}Y U) \right)^{-k'/2} \ll \summ_U \left( 1+ \ymin\tr( \tyd^{-1} \tp{U} U) \right)^{-k'/2} .
  \end{align}
  Here we have used $Y\ge \ymin 1_n$ where $\ymin$ is the smallest diagonal eigenvalue of $Y$.

  In \eqref{usum1}, we introduce a new variable $W$ such that $\tp{U}U=W$. Note that $W$ lies in $\Sym \z^+ \cap \SL \z$. The new summation thus becomes
  \begin{align} \label{usum2}
    \mbb S \ll_n \summ_W \# \{U \mid \tp{U}U=W \} \left( 1+ \ymin\tr( \tyd^{-1} W) \right)^{-k'/2}.
  \end{align}
For the counting function in \eqref{usum2}, which we will call $h(W)$, we note that if $U_0=U^{-1} \in \GL \z$ is such that $\tp{U_0}WU_0=1_n$ and so
  \begin{align}
    h(W)=\# \{U \mid \tp{(UU_0^{-1})}UU_0^{-1}=W_1 \}= h(1_n),
  \end{align}
  so we can \textsl{assume that $W$ is reduced}. This clearly implies that
  \begin{align} \label{usum3}
    h(W) \ll_n  1.
  \end{align}
  Then from \eqref{usum2} and \eqref{usum3}, we have 
  \begin{align} \label{usum4}
    \mbb S \ll_n \summ_W \left( 1+ \ymin\tr( \tyd^{-1} W) \right)^{-k'/2}
    \ll \summ_W \left( 1+ \ymin\sumn_{j=1}^n \mu_j^{-1} w_j \right)^{-k'/2}.
  \end{align}
We execute the sum over $W$ by grouping together $W$ such that $\prod_j w_j =m$ for integers $m \ge 1$. Let us define the higher divisor function $\sn(m)$
and the number theoretic function $g(\ut)$ where $\ut=  (t_1,t_2,\cdots, t_n) \in \N^n$ by
\begin{align}
\sn(m) &= \# \left\{ (s_1,s_2,\ldots, s_n)\in \N^n \midmid \prod\nolimits_j s_j =m \right\},\dn\\
g(\ut)& = \# \left\{ V \in \SL \z \cap \Sym \z^+ \midmid  V_{D} = \dia(t_1,t_2,\ldots, t_n)  \right\},
\end{align}
where $V_{D}$ denotes the diagonal matrix with the diagonal elements of $V$. It is well-known (see e.g. \cite{tenen}) that $\sn(m) \ll _\epsilon m^\epsilon$ and that $g(\ut) \ll (\prod_j t_j)^{(n-1)/2}$. The latter can be checked as follows.

If $V = (v_{i,j}) \in g(\ut)$, then $v_i := v_{i,i} = t_i$ for all $1 \le i \le n$. Further since $V>0$, one has $2|v_{i,j}| \le (t_i t_j)^{1/2}$. Since $t_i \ge 1 $ for all $i$ this implies that the number of choices for $V$ is 
\begin{align}
\ll \prodd_{i<j}(t_i t_j)^{1/2} = (\prodd_j t_j)^{(n-1)/2}.
\dn\end{align}

  Then \eqref{usum4} can be written as
  \begin{align}
    \mbb S &\ll_n  \summ_{m=1}^\infty \, \summ_{\ut \in \N^n, \,  \prod_j t_j=m }
    g(\ut) \big( 1+ \ymin\summ_{j=1}^n \mu_j^{-1} t_j \big)^{-k'/2}.
  \dn\end{align}
  Applying the A.M.-G.M. inequality to the sum over $j$ above, we get
  \begin{align}
    \mbb S
      & \ll_n \ymin^{-\frac{nk'}{2(n+1)}}\summ_{m=1}^\infty \,\, \summ_{\ut \in \N^n, \, \prod_j t_j=m }  (\prodd_j t_j)^{(n-1)/2} (\prodd_j \mu_j )^{\frac{k'}{2(n+1)}} (\prodd_j t_j)^{-\frac{k'}{2(n+1)}} \dn\\
      & = \left(\tyd/\ymin^n\right)^{\frac{k'}{2(n+1)}} \summ_{m=1}^\infty  m^{-\frac{k'}{2(n+1)} + \frac{n-1}{2}} \sn(m) \dn\\
      & \ll_{n, \epsilon} \left(\ty_1/\ymin^n\right)^{k'/(2(n+1))} \summ_{m=1}^\infty  m^{(-k'+n^2-1)/(2(n+1)) +\epsilon} \ll_n \left(\ty/\ymin^n\right)^{k'/(2(n+1))},
  \end{align}
provided $(k'-n^2+1)/(2(n+1))>1+\epsilon$ i.e., $k' > (1+\epsilon)(n+1)^2$. Thus, as mentioned at the onset of the proof, we choose
\begin{equation}
 k'= (n+1)^2 +\epsilon. 
\dn\end{equation} 
With this choice, we see that $\mbb S \ll (\ty/\ymin^n)^{(n+1)/2 +\epsilon}$. Putting this in \eqref{usplit} finishes the proof.
\end{proof}

\begin{thm} \label{bkbd}
The Bergman kernel $\bkzz$, for $k>(n+1)^2$ and for all $Z\in\fn$ is bounded by
\begin{align}
\bkzz \ll_{n, \epsilon}
\begin{cases}
  k^{3n(n+1)/4} \det(Y)^{(n+1)/2 +\epsilon}y_1^{-n(n+1)/2 - n\epsilon} \\
  k^{n(n+1)/2}\det(Y)^{3(n+1)/4+\epsilon}y_1^{-n(n+1)/2 - n\epsilon} ,
\end{cases}
\end{align}
where $y_1$ is the smallest diagonal entry of $Y$. The above result can also be stated in terms of $\ymin$, the smallest eigenvalue of $Y$ in place of $y_1$.
\end{thm}

\begin{proof}
We look at the expression for $\bkzz$ from \eqref{eq:h_gamma_breakdown} for $Z \in \mc F_n$. We only explain the first of the above bounds, and the rest is similar.
For the sum over $T$ and $U$ respectively, we use the bounds from \lemref{lipcor} and \lemref{headache}. Namely, after executing the sum over $T$ as in \lemref{lipcor}, we apply \lemref{headache} with appropriate $k$ and take $\ty = \Im(M\langle Z \rangle)$. Then notice that the final bound for the sum over $U$ (as in \eqref{eq:h_gamma_breakdown}) \lemref{headache} becomes
\begin{equation}
\sumn_U \det(\ty + \tp{U}YU)^{-k} \ll 2^{-nk} \frac{ \det(Y)^{-k+(n+1)/2+\epsilon} \ymin^{-n(n+1)/2 - n\epsilon}} { |\det(CZ+D)|^{-k+n+1+\epsilon} }.
\dn\end{equation}

Again looking at \eqref{eq:h_gamma_breakdown}, we see that the exponential factor and the power of $\det(Y)^k$ cancel off -- and the remaining sum over the co-prime symmetric pairs $\{C,D\}$ converges:
\begin{align}
\bkzz \ll k^{3n(n+1)/4} \det(Y)^{(n+1)/2+\epsilon} \ymin^{-n(n+1)/2 - n\epsilon}\cdot \sumn_{\{C,D\}} \frac{1}{|\det(CZ+D)|^{n+1+\epsilon}}  .
\end{align}
That the sum over $\{C,D\}$ -- $\mbb E_s(Z):= \sum_{\{C,D\}} |\det(CZ+D)|^{-s} $ (being a majorant of the Siegel Eisenstein series) converges for $s>n+1$ is well-known (see e.g. \cite{freitag1983siegel}). But we also need to know that it can be bounded independently of $Y$.

To see this, we define the region $W(\delta)=\{ Z \in \hn \mid \tr(\Re(Z)^2) \le \delta^{-1}, \, \Im(Z) \ge \delta 1_n  \}$. Such regions are often used to establish analytic properties of automorphic forms. We
divide $\mc F_n$ into two parts -- $K=W(2/n) \cap \mc F_n$ and $V=\mc F_n -K$. Note that $K=\{ Z \in \mc F_n \mid \det(Y) \ge 2/n\}$.

In $K$ we invoke \cite[p.~68, $5.4_2$~Hilfssatz]{freitag1983siegel} to $Z_0=i1_n$ to obtain that for all $Z \in K$, one has $|\det(CZ+D)| \ge c |\det(C i1_n +D)|$ for some constant $c>0$ depending only on $n$. Then clearly $\mbb E_{n+1+\epsilon}(Z) \ll \mbb E_{n+1+\epsilon}(i 1_n) $ and we are done.

For the region $V$, note that $V \subset K'=\{ Z \in \mc F_n \mid \det(Y) \le 2\}$. Since $K'$ is compact, the continuous function $\mbb E_{n+1+\epsilon}(Z)$ attains a maximum $Z_*$ in $K'$ which depends only on $n$.

Because $Y$ is positive definite, we have $\ymin\le y_1$ and from the reduction conditions \eqref{reduct-implication}, we have $y_1\le r_n\,\ymin$. So the two can be swapped freely.
For the second bound, we proceed exactly as above, but with the second bound from \lemref{lipcor} and using \lemref{headache}. This finishes the proof.
\end{proof}

\begin{cor} \label{f-bkbd}
Let $Z \in \mc F_n$, $k >n(n+1)/2$ and consider $F \in S^n_k$ with $\norm{F}_2=1$. Then
\begin{equation} \label{eq-f-bkbd}
\det(Y)^{k/2} |F(Z)| \ll_{n,\epsilon}
\begin{cases} 
  k^{\frac{3n(n+1)}{8}} \det(Y)^{\frac{n+1}{4}+\epsilon}\\
  k^{\frac{n(n+1)}{4}} \det(Y)^{\frac{3(n+1)}{8}+\epsilon}.
\end{cases}
\end{equation}
\end{cor}

\begin{rmk} \label{bk-abs-sum}
  If we follow the proof of \thmref{bkbd} but use \lemref{lipz} instead of \lemref{lipfo}, then the absolute sum is bounded by
  $\sum_{\gamma\in\Gamma_n}|h_\gamma(Z)|^k
  \ll_{n,\epsilon } \det(Y)^{n+1+\epsilon}$ for $Z\in\fn$ and $k>(n+1)^2$.
\end{rmk}

\subsection{The case \texorpdfstring{$n=1$}{n1}: bound for the sup-norm using minimal hypotheses} \label{n=1exam}
When $n=1$, we of course have H. Xia's sharp result \cite{xia2007norms} that for a Hecke eigenform $f \in S_k$ with $\oldnorm{f}_2=1$, $ k^{1/4 - \epsilon} \ll_\epsilon \oldnorm{f}_\infty \ll_\epsilon k^{1/4 +\epsilon}$ for any $\epsilon>0$. This result is reliant on Deligne's bound and deep properties of zeros of $L$-functions, namely the non-existence of Landau-Siegel zeros for the symmetric-square $L$-function attached to $f$. This directly implies $\bkzz\ll k^{3/2+\epsilon}$. 
Here our aim is modest: we want to illustrate a simpler way of proving the bound $ \bkzz \ll  k^{3/2}$. This bound has been obtained by Kramer et al. \cite{kramer1} by analysing the heat kernel corresponding to the weight-$k$ Laplacian on $\mc H_1$. 

Keeping in view of our limited information in higher degrees, the point is that we would only use
the Hecke bound on Fourier coefficients along with a somewhat careful Bergman kernel analysis. This might give some hope of bounding the Bergman kernel in higher degrees to obtain a reasonable sup-norm bound (viz. $\oldnorm{F}_\infty \ll_\epsilon k^{3n(n+1)/8 +\epsilon}$) and seems to us the most promising approach to be considered in the future. We first recall a simple lemma.

\begin{lem}
Let $\alpha, d>0$.
\begin{align} \label{1dim}
\sumn_{t \ge 1} t^{\alpha-1} \expn{-td} \ll d^{-\alpha+1} \Gamma(\alpha) (d^{-1} + (\alpha-1)^{-1/2}).
\end{align}
\end{lem}

\begin{proof}
For this note that the function $h(x):=x^{\alpha-1} \expn{-xd}$ attains its maximum at $x=(\alpha -1)/d$ and so the l.h.s. of \eqref{1dim} is bounded by
\begin{align}
\int_0^\infty & x^{\alpha-1} \expn{-xd} dx + 2 h(d/(\alpha -1))= d^{-\alpha} \Gamma(\alpha) + 2 \left(\frac{\alpha -1}{d} \right)^{\alpha -1} \expn{-\alpha+1}
\end{align}
which is bounded as
$\ll d^{-\alpha+1} \Gamma(\alpha)( \frac{1}{d} + \frac{1}{(\alpha-1)^{1/2}} )$
up to an implied absolute constant.
\end{proof}

From the Petersson formula and bounding the $n$-th Fourier coefficient of the Poincar\'e series $P_{n}$ with the trivial bound for Kloosterman sums and Bessel functions (i.e., $J_{k-1}(x) \le \min\{ 1, (x/2)^{k-1}\}$ for all $x>0$), we have the following uniform bound for any orthonormal basis $\basis[1]$ of $S_k$:
\begin{equation}
  \sqrt{\summ_{f\in\basis[1]}|a_f(n)|^2}
  \le \frac{(4 \pi n)^{k/2} }{\Gamma(k-1)^{1/2}}.
\dn\end{equation}
Then from the lemma above with $\alpha=1+k/2$ and $d=2 \pi y$, and \lemref{bkconnect}, we get
\begin{align} 
\sqrt{\bkzz} \ll \frac{\Gamma(k/2+1)\, (y^{-1}+k^{-1/2})\, (4 \pi)^{k/2}}{\Gamma(k-1)^{1/2} \, (2 \pi)^{k/2}} \ll k^{5/4} (y^{-1}+k^{-1/2}) \ll k^{3/4} \label{bkzn1bd2}
\end{align} 
in the region $\{ z \in \mc F_1 \mid y \ge k^{1/2} \}$.

In the complementary region, we look at the Bergman kernel directly. We already know the shape of the bound: for any $z \in \mc F_1$,
\begin{align}
\frac{ \bkzz }{k} & \ll \sum_{\gamma \in \Gamma_1 \colon \abs{\gamma(z)-z}>\sqrt{\delta}y } |h_\gamma(z)|^k +\sum_{\gamma \in \Gamma_1 \colon \abs{\gamma(z)-z} \le \sqrt{\delta}y } |h_\gamma(z)|^k \n \\
& \ll \max_{\gamma \in \Gamma_1 \colon \abs{\gamma(z)-z}> \sqrt{\delta}y} |h_\gamma(z)|^{k-k_0} \big( \sum_{\gamma \in \Gamma_1} |h_\gamma(z)|^{k_0}  \big) + \# \{\gamma \in \Gamma_1 \colon \abs{\gamma(z)-z} \le \sqrt{\delta}y \} \n \\
& \ll (1+\delta)^{-k+k_0} y^{2} + X_\delta, \label{bkzn1bd}
\end{align}
where $X_\delta = \# \{\gamma \in \Gamma_1 \colon |\gamma(z)-z| \le \sqrt{\delta}y \}$ and use \rmkref{bk-abs-sum}. For the bound on $\max |h_\gamma(z)|$ we have used \eqref{hdecay} (see also \cite[Lemma~1]{cogdell2011bergman}, keeping in mind the dependence on $y$).
We now estimate the size of $X_\delta$ carefully. We follow the setting of section~\ref{bkzsec} and \cite[section~3]{das2015supnorms}. We make things explicit because it is possible and useful, even though this may be a little repetition to section~\ref{bkzsec}.

Let us recall that the condition $|\gamma(z)-z|\le \alpha$ implies that
$\abs{\gamma_0^{-1} \gamma \gamma_0 - \mathsf{k}} \ll (\alpha/y)^{1/2}$. Here $\mathsf{k}_z = \gamma_0 \mathsf{k} \gamma_0^{-1}$, where $\gamma_0 = \psmb y^{1/2} & xy^{-1/2} \\ 0 & y^{-1/2} \psme$. For us, $\alpha= \sqrt{\delta}y$. Thus 
$\abs{\gamma_0^{-1} \gamma \gamma_0 - \mathsf{k}} \ll \delta^{1/4}$, i.e., $\gamma_0^{-1} \gamma \gamma_0= \mathsf{k} + O(\delta^{1/4})$.

We see that if we put coordinates $\gamma = \psmb  a& b \\ c & d \psme$, $\mathsf{k} = \psmb p & q \\ -q & p \psme$ such that $p^2+q^2=1$, then
from the analysis done in section~\ref{bkzsec} (see also \cite[section~3]{das2015supnorms})
\begin{align} 
\begin{pmatrix}
a & b \\ c & d
\end{pmatrix} & = \gamma_0 \begin{pmatrix}
p & q \\ -q & p
\end{pmatrix} \gamma_0^{-1} + O \left( \gamma_0  O (\delta^{1/4}) \gamma_0^{-1} \right) \n \\
& = \begin{pmatrix}
p-qx/y & qx^2/y+qy \\ - q/y & qx/y+p
\end{pmatrix} + O
\begin{pmatrix}
\delta^{1/4} & \delta^{1/4} y 
\\ \delta^{1/4}/y & \delta^{1/4}
\end{pmatrix}. \label{n1count2}
\end{align}
The $O(\cdots)$ term can be seen simply by "pretending" that the quantities $p,q$ were $O(\sqrt{\delta}y)$.

Our choice for $\delta$ is $\delta = \delta_0 :=k^\epsilon/k$. With this choice, we see that the number of choices of $a,c,d$ in \eqref{n1count2} is absolutely bounded. Further the number of choices for $b$ is at most $O(y)$. Since we are considering the region $\{ z \in \mc F_1 \mid y \ll k^{1/2} \}$, we get
\begin{align} \label{xbd}
X_{\delta_0} \ll k^{1/2}.
\end{align}
The first term in \eqref{bkzn1bd} decays exponentially in $k$.
Thus from \eqref{bkzn1bd2}, \eqref{bkzn1bd} and \eqref{xbd} we conclude that for all $z \in \mc F_1$,
\begin{equation}
\bkzz[z]=\sumn_f y^k|f(z)|^2 \ll  k^{3/2}.
\end{equation}

\section{Amplification: over compact sets} \label{ampl}
In this section, we want to go beyond the preliminary bound on the Bergman kernel over a compact subset of $\fn$ and prove a power saving bound for the same. For this, we would use the setting of the amplification method considered in \cite{blomer2016supnorm} and \cite{das2015supnorms} along with some of our earlier results.

Let $\Omega \subset \fn$ be a compact set which does not depend on $k$. By the results of \cite{cogdell2011bergman}
we see that for any $\Omega \subset \fn^\circ$ ($\fn^\circ$ being the interior of $\fn$) as above, and some $\delta>0$,
\begin{align} \label{hkconv1}
\bkzz = 2^{-1} \ank R_k(Z) = \ank + O(e^{-\delta k}) \asymp_n k^{n(n+1)/2}.
\end{align}
This is actually worked out in a few lines before section~5 in \cite{cogdell2011bergman}. Otherwise, and also generally one argues that
\begin{align} \label{hkconv2}
\bkzz\ll  \ank \summ_{\gamma \in \Gamma_n} |h_\gamma(Z)|^k \ll  \ank \summ_{\gamma \in \Gamma_n} |h_\gamma(Z)|^{2n+1} \ll_\Omega \ank,
\end{align}
since by \cite[Lemma~1]{cogdell2011bergman}, $|h_\gamma(Z)| \le 1$ and by Godement's theorem as given in \cite[p.~79, Chap.~6, Prop.~2(iii)]{klingen1990siegel}, for all $\kappa >2n$ ($\kappa n$ even), the series 
\begin{equation} \label{godement}
\det(Y)^{\kappa/2} \summ_{\gamma \in \Gamma_n} |h_\gamma(Z)|^{\kappa}
\end{equation}
converges uniformly for all $Z \in \Omega$ as above. 
However, $\det(Y)$ is bounded below (even absolutely, since $Z \in \fn$), and so we have \eqref{hkconv2}. This can also be derived easily from the bound in \rmkref{bk-abs-sum}. Dropping all but one term in \eqref{hkconv2} gives the `generic' or the trivial bound
\begin{equation} \label{supcomp1}
    \oldnorm{F}_\Omega = \sup_\Omega \, \det(Y)^{k/2} |F(Z)| \ll_n k^{n(n+1)/4}.
\end{equation}
This also follows from the second bound in \corref{f-bkbd}. Our aim is to improve \eqref{supcomp1}.

Henceforth in this subsection let us assume that $F$ is an $L^2$ normalised Hecke eigenform which is not a Saito-Kurokawa (SK) lift. (SK lifts will be discussed separately, see \cite{das-anamby}.) Accordingly, we embed $F$ in a Hecke (ortho-normalised) basis, say $\mc B^*_k$.
Our choice of the amplifier is the same as constructed in \cite{blomer2016supnorm}. We quote the inequality at the heart of the amplification:
\begin{align} \label{hreln1}
|\lambda(p,F)| + \frac{1}{p^{3/2}}  |\lambda(p^2,F)| +  \frac{1}{p^{9/2}}  |\lambda(p^4,F)| \gg p^{3/2},
\end{align}
which says that the three eigenvalues above cannot be simultaneously small. For the reader's convenience we recall a proof of \eqref{hreln1}. Recall that the Ramanujan conjecture at finite places is known for $F$ by the work of Weissauer \cite{weis}. In particular this implies that $x=|\lambda(p,F)|, y=|\lambda(p^2,F)|, z=|\lambda(p^4,F)|$ are $O(p^{3r/2} )$ for $r=1,2,4$ respectively and the implied constant is absolute. 
For future reference note the normalisation
\begin{equation}
\lambda(n,F) = \lambda_n(F) n^{3/2},
\dn\end{equation}
where $\lambda_n(F)$ denotes the normalised eigenvalues which satisfy $\lambda_n(F) \ll_\epsilon n^\epsilon$.
From the Hecke relation (see \cite[p.~1011]{blomer2016supnorm} or \cite{andrianov1974euler}) we see (using the bounds for $x,y,z$ from above) that
\begin{align}
p^6 \le (p^2+2p^3)x^2 + x^4 +p^2y +y x^2 + y^2 +z \ll p^{9/2}x + p^3 y +z.
\dn\end{align}

Fix an $F_0$ in $\mc B^*_k$ which is not a SK lift.
Let now $F \in \mc B^*_k$ be arbitrary. Consider a parameter $L \gg 1$ (to be specified  later), and define $\mc P$ to be the set of primes in $[L,2L]$. Put also, following \cite{blomer2016supnorm}, $x(n):= \mrm{sgn}(\lambda(n,F_0))$.
\begin{align}
A_F:= (\sum_{p \in \mc P} x(p) \lambda(p,F))^2 + (\sum_{p \in \mc P} x(p^2)p^{-3/2} \lambda(p^2,F))^2 + (\sum_{p \in \mc P} x(p^4) p^{-9/2} \lambda(p^4,F))^2  \ge 0.
\end{align}

We then consider the expression
\begin{align} \label{azw1}
 A(Z,W):= \summ_{F \in \mc B^n_k(H)} A_F \cdot F(Z) \overline{F(W)}
\end{align}
and rewrite it in terms of the Bergman kernel $B_k(Z,W)$. First let us note that the quantity 
\begin{align} \label{azw2}
\mbb A(Z,W) := \summ_{F \in \mc B^n_k(H)} A_F \cdot \det(Y)^k F(Z) \overline{F(W)}
\end{align}
satisfies, upon putting $Z=W$ and using the prime number theorem:
\begin{align} \label{ampl1}
\mbb A(Z,Z) \ge A_{F_0} \det(Y)^k |F_0(Z)|^2 \gg \frac{L^5}{(\log L)^2} \det(Y)^k |F_0(Z)|^2.
\end{align}

From \eqref{azw1} we can write, after expanding the amplifier $A_F$ that
\begin{align}
A(Z,W)=  \sum_{r=1,2,4}\Bigg( \sum_{p_1 \neq p_2} \frac{ x(p^r_1 p^r_2) }{(p_1 p_2)^{3(r-1)/2}}
& \sum_F  F(Z) \overline{ (F | T_{p^r_1 p^r_2})(W) }  \n \\
& + \sum_p \frac{1}{p^{3(r-1)}} \sum_F F(Z) \overline{F| (T_{p^r})^2(W)}
\Bigg).\dn\end{align}
But from the definition of the Bergman kernel $B_k(Z,W)=\sum_F F(Z) \overline{F(W)}$ and that of the Hecke operator $T_n$ given by
\begin{align}
F | T_m = \sum _j F | M_j, \q (S(m) = \bigcup_j \Gamma_2 M_j)
\dn
\end{align}
where $S(m)=\{M\in\mrm{GSp}^+(2,\z)\mid\tp{M}\smat{0}{-1_n}{1_n}{0}M=m\smat{0}{-1_n}{1_n}{0}\}$
and where we have put $F | M_j := F | \det(M_j)^{-1/4} M_j$. This gives us
\begin{align}
\sum_F F(Z) \overline{F|T_m(W)}& = B_k(Z,W)|^{(W)}T_m = 2^{-1} \ank[2] \sum_j \sum_{\gamma \in \Gamma_2} \det\left( \frac{Z-\overline{W}}{2i} \right)^{-k} \Big|^{(W)} \gamma M_j \\
& =  2^{-1} \ank[2] \sum_{g \in S(m)} \det\left( \frac{Z-\overline{W}}{2i} \right)^{-k} \Big|^{(W)} g =: B^m(Z,W).
\dn\end{align}
Then we have
\begin{align} \label{ampl2}
A(Z,W) \ll  \sum_{r=1,2,4} \sum_{p_1 \neq p_2} (p_1 p_2)^{3(1-r)/2} |B^{p_1^r p_2^r}(Z,W)| + \sum_{0 \le s \le 4} \sum_{p \in \mc P} p^{3-2s}  |B^{p^{2s}}(Z,W)|
\end{align}

Our next goal is to estimate the quantity 
\begin{equation}
\mbb B^m(Z,W):=\det(Y)^k B^m(Z,W)
\dn\end{equation}
when $Z=W$.
For $\gamma \in S(m)$, let us put $\tg = m^{-1/4} \gamma \in \sptwo$. Then we start with 
\begin{align} \label{hndiv}
2 \ank[2]^{-1} \mbb B^m(Z,Z) = \sum_{\gamma \in S(m)\colon \norm{\tg(Z)-Z}>\delta_0} h_\gamma(Z)^k +\sum_{\gamma \in S(m)\colon \norm{\tg(Z)-Z} \le \delta_0} h_\gamma(Z)^k ,
\end{align}
for some $\delta_0 >0$ to be specified later. 
Then the l.h.s. of \eqref{hndiv} is (with $k_0$ depending only on $n$)
\begin{align} \label{hndiv2}
\le \left( \max_{\gamma \in S(m)\colon \norm{\tg(Z)-Z}>\delta_0} |h_\gamma(Z)|^{k-k_0} \right) \sum_{\gamma \in S(m)\colon \norm{\tg(Z)-Z}>\delta_0} |h_\gamma(Z)|^{k_0}   +\mbb S_m(Z, \delta_0), 
\end{align}
where we have put
\begin{equation} \label{smdef}
\mbb S_m(Z, \delta_0)  :=\# \{\gamma \in S(m)\colon \norm{\tg(Z)-Z} \le \delta_0 \}
\end{equation}
Since the action of $\Gamma_2$ on $\mc H_2$ is discontinuous, we see that the number of terms in the second summation is finite. We will make this more precise in the next section.

\subsection{Counting points in \texorpdfstring{$\mbb S(m)$}{sm}} \label{bkzsec}
This subsection is partly inspired from \cite{cogdell2011bergman}, \cite{das2015supnorms}.

For a matrix $A \in \m \R$ we put $\norm{A}$ to be the $L^2$ norm of $A$, i.e., $\norm{A}^2 = \sum_i\sum_j a_{i,j}^2$. Recall that such a norm satisfies $\norm{AB} \le \norm{A} \norm{B}$.

First of all from \cite[Lemma~1]{cogdell2011bergman}, we know that $|h_\gamma(Z)| \le 1$ for all $\gamma \in \sptwo$. Furthermore, from the same lemma, given any $\rho>0$, there exist an absolute constant $c$ such that
\begin{align} \label{hdecay}
|h_\gamma(Z)| \le (1+\rho^2)^{-1/2} \text{    provided     } \norm{\gamma(Z)-Z} >c \, \rho\, \tr(Y).
\end{align}
Accordingly, we define $\rho$ (with $c$ as above), to be specified later:
\begin{equation}
\delta_0=: c \, \rho \tr(Y).
\dn\end{equation}

Throughout the rest of this subsection, we put $Y^{1/2}$ to be the unique positive definite (symmetric) square-root of $Y$ and define
\begin{equation}
\gamma_0 = \begin{pmatrix}
Y^{1/2} & XY^{-1/2} \\ 0 & Y^{-1/2}
\end{pmatrix} \in \Sp[2]{\R},
\dn\end{equation}
so that $\gamma_0(i1_2)=Z$.

\begin{lem} \label{gamma0}
With $\gamma_0$ as above, and any $Z_1,Z_2 \in \hn$, one has
\begin{equation}
\norm{\gamma_0(Z_1)-\gamma_0(Z_2)} \le a \implies \norm{Z_1-Z_2} \le a \, \tr(Y^{-1}).
\end{equation}
\end{lem}

\begin{proof}
One observes that $\gamma_0(Z_j)=Y^{1/2} Z_j Y^{1/2} + X$ for $j=1,2$ and so the hypothesis gives us $\norm{Y^{1/2} W Y^{1/2}} \le a$; where we have put $W=Z_1-Z_2$.

Recall that $\norm{A}^2 = \tr(A \tp{\overline{A}})$. Let $P$ be an orthogonal matrix diagonalizing $Y$: $PY\tp{P} = D$ (say) with eigenvalues $d_1,\cdots, d_n$. Then
\begin{align}
\norm{ Y^{1/2} W Y^{1/2} }^2 = \tr (YWY \overline{W}) = \tr(D W_P D \overline{W}_P),
\dn\end{align}
where we have put $W_P=\tp{P}WP$. Then it follows that
\begin{equation}
\sum_i d_i^2 s_i \le a^2, \text{ where } s_i = \sum_{j=1}^n |W_{i,j}|^2
\dn\end{equation}
Now it is easy to see that $\norm{W}^2=\norm{W_P}^2 = \sum_i s_i \le a^2(\sum_i d_i^{-2}) = a^2 \, \tr(Y^{-1})^2$ from which the lemma follows.
\end{proof}

Our next aim is to show that for small enough $\delta_0$, the condition $\norm{\gamma(Z)-Z} \le \delta_0$ ($\gamma \in \Sp{\R}$) implies that $\gamma$ is `close' to the maximal compact subgroup $ K_Z = \gamma_0 K \gamma_0^{-1}$ fixing $Z$, in a suitable sense. For this we recall the Iwasawa decomposition of $\Sp[n] \R$: for $g \in \Sp[n] \R$, 
\begin{align}
g=\mc N \mc A \mc K; \q \mc N= \begin{pmatrix}
B & C \\ 0 & {B^{t}}^{-1}
\end{pmatrix}, \mc A= \begin{pmatrix}
L & 0 \\ 0 & L^{-1}
\end{pmatrix}, \mc K=\begin{pmatrix}
P & Q \\ - Q & P
\end{pmatrix}; 
\dn\end{align}
where $B$ is unipotent and upper triangular, $L= \dia (e^{t_1}, \cdots, e^{t_n})$, and $\mc K \in K=\SO[2n]{\R}\cap\Sp{\R}$.

Put $\tgg= \gamma_0^{-1} \tg \gamma_0$ so that $\norm{\gamma(Z)-Z} = \norm{\gamma_0 \tgg(i1_n)- \gamma_0(i1_n)} \le \delta_0$ implies, by \lemref{gamma0} that
\begin{align} \label{na}
\norm{ \tgg (i1_n)- i1_n} = \norm{\mc N \mc A (i1_n) - i1_n} \le \delta_0 \, \tr(Y^{-1}) =: \eta.
\end{align}
Here we have used the Iwasawa decomposition for $\tgg$, along with th fact that $\mc K$ fixes $i1_n$. We see that
\begin{align}
\eta^2 \ge \norm{C\tp{B} +i (BL^2\tp{B} -1_n) }^2 = \tr\left( (N+iM)(N-iM) \right) = \tr(N^2)+\tr(M^2);
\dn\end{align}
where we have put $N=C\tp{B}$ and $M = (m_{i,j})= BL^2\tp{B} -1_n$.

\textsl{From now on, we will assume for simplicity that $n=2$, but this part of the analysis, up to the end of this subsection, carries over to any $n$.}

Further put $G=(g_{i,j})=BL$, so that 
\[ M=G\tp{G} - 1_2=\smat{g^2_{1,1}+g^2_{1,2}-1}{g_{1,2}g_{2,2}}{g_{1,2}g_{2,2}}{g^2_{2,2}-1}.\] 
Then $\tr (M^2) \le \eta^2$ shows that $\sum_i \sum_j m_{i,j}^2 \le \eta^2$. In particular $m_{i,i}^2 \le \eta^2$.
Thus we have for all $1 \le i \le 2$,
\begin{align}
\abs{m_{i,i}}=|\summ_{j\ge i} g_{i,j}^2 -1 | \le \eta \implies 1-\eta \le \summ_{j\ge i} g_{i,j}^2  \le 1+\eta.
\dn\end{align}

Note that $G$ is upper triangular, and $g_{i,j}= b_{i,j}l_j$ for  $j \ge i$ and $0$ otherwise, where obviously $B=(b_{i,j})$ and $l_j=e^{t_j}$ are the entries of $L$. We have $b_{1,1}=b_{2,2}=1$. Then from $m_{2,2}= g^2_{2,2}-1= l^2_{2}-1$ we see that 
$1-\eta \le l^2_2 \le 1+\eta$. Moving on, next consider $m_{1,2} = g_{1,2}g_{2,2}$. From $|m_{1,2}| \le \eta$ 
we infer that $|g_{1,2}| \le \eta/(1-\eta)$. Since $g_{1,2}=b_{1,2}l_2$, we get $|b_{1,2}| \le  \eta /(1-\eta)^{3/2} \ll \eta$ for $\eta$ small enough. Finally from $|m_{1,1}| \le \eta$ we get $l_1=g_{1,1}$
satisfies $1-2\eta \le l^2_1 \le 1+2\eta$ if $\eta \le 2/3$.

Now $C=N{B^{t} }^{-1}$ and so $\norm{C} \le \norm{N} \norm{B^{-1}} \ll \eta$.
In our notation now note that
\begin{equation}
\mc N \mc A - I_4= \begin{pmatrix}
BL & CL^{-1} \\ 0 & {\tp{B}}^{-1} L^{-1}
\end{pmatrix} - I_4= \begin{pmatrix}
G -I_2  & CL^{-1} \\ 0 & {\tp{G}}^{-1} - I_2
\end{pmatrix}
\dn\end{equation}
which, by the above analysis implies that $\norm{\mc N \mc A-I_4} \ll \eta^{1/2}$
for $\eta$ small enough. This is easy to see since $G = \smat{l_1}{g_{1,2}}{0}{l_2}$ and with $\eta$ small enough, $l_1,l_2$ are close to $1$ whereas $g_{1,2}$ is close to $0$. Moreover, $C$ is also close to $0$. Thus $\mc N \mc A=I_4 +O(\eta^{1/2})$. Multiplying by $\mc K$ on both sides gives $\tgg = \mc K + O(\eta^{1/2})$.
Conjugating with $\gamma_0$ gives (recall $K_Z= \gamma_0 K \gamma_0^{-1}$ and \eqref{na})
\begin{align} \label{etaconf}
\norm{\tg - K_Z} \ll \eta^{1/2} = (c \, \rho \, \tr(Y) \tr(Y^{-1}) )^{1/2} \q \q (\delta_0=: c \, \rho \tr(Y)).
\end{align}
From \eqref{hndiv2} and the estimate \eqref{hdecay} we can write 
\begin{align} \label{hnest}
2 \ank[2]^{-1} \mbb B^m(Z,Z) \ll (1+\rho^2)^{-k+k_0}\sumn_{\gamma \in S(m)} |h_\gamma(Z)|^{k_0}  + \#\mathscr S_m(Z, u \rho^{1/2}),
\end{align}
where $\mathscr S_m(Z, \delta) :=\{\gamma \in S(m)\mid \norm{\gamma-K_Z}\le \delta\}$ and $u=\max_{Z\in\Omega}(c \tr(Y) \tr(Y^{-1}) )^{1/2} $, which depends only on $\Omega$.

\subsection{Back to the proof} 
To bound the quantity $\sum_{\gamma \in S(m)} |h_\gamma(Z)|^{k_0}$ we use the method in \cite{das2015supnorms}. Unfolding, we can write
\begin{align}
\summ_{\gamma \in S(m)} |h_\gamma(Z)|^{k_0} \le \summ_j (\summ_{\gamma \in \Gamma_2} |h_{\gamma M_j}(Z)|^{k_0}).
\dn\end{align}

Put $h_\gamma(Z,W)^{k_0} = \det \left( \frac{Z-\overline{W}}{2i} \right) \lvert^{(W)}_{k_0} \gamma$.
From \cite[p.~79,81]{klingen1990siegel} (Godement's theorem) we know that (with $V=\Im(W)$) for $Z,W \in \Omega$,
\begin{align}
\summ_{\gamma \in S(m)} |h_\gamma(Z,W)|^{k_0}
  &= \summ_j \summ_{\gamma \in \Gamma_2} \left.\abs{h_\gamma(Z,W)}^{k_0}\right|_{k_0}^{(W)} M_j
  \ll_\Omega \summ_j \abs{\det(V)^{-k_0/2} |_{k_0}M_j } \n \\
& \ll \# \{ \Gamma_2 \backslash S(m) \} , \label{hzw2}
\end{align}
since $\det(V)^{-k_0/2}$ is invariant under the $|_{k_0}$ action of $\sptwo$ in absolute value.

The following lemma perhaps is known, but since we could not find a proof in the literature, so we include it here for convenience. 
\begin{lem} \label{coset-count}
With the above notation,
$\# \{ \Gamma_2 \backslash S(m) \} \le m^{3} \sigma_0(m)^2$.
\end{lem}
\begin{proof}
From \cite{andrianov1974euler} or \cite{kodama}, we see that a set of right coset representatives of the set $\Gamma_2 \backslash S(m)$ can be chosen in the form $\psmb A& B \\ 0 & C \psme $ ($A,B,D \in \mtwo$) where
\begin{align}
& A\tp{C}=mI_2, A\tp{B}=B\tp{A}, \q A=(a_{i,j}),B=(b_{i,j}), C=(c_{i,j}), \n\\
& a_{2,1}=0, 0 \le a_{i,j} < a_{j,j} \text{ if } i<j; \q 0 \le b_{i,j} <c_{j,j} \text{ if } i \le j.
\dn\end{align}
From the equality $A\tp{D}=mI_2$ we get
\begin{align}
\begin{pmatrix}
a_{1,1}c_{1,1}+a_{1,2}c_{1,2} & a_{1,1}c_{2,1}+a_{1,2}c_{2,2} \\
a_{2,2}c_{1,2} & a_{2,2} c_{2,2} 
\end{pmatrix} = \begin{pmatrix}
m & 0 \\ 0 & m
\end{pmatrix},
\dn\end{align}
which shows that $a_{2,2}c_{1,2}=0=a_{1,1}c_{2,1}+a_{1,2}c_{2,2}$; $a_{2,2} c_{2,2}=m=a_{1,1}c_{1,1}+a_{1,2}c_{1,2}$. Clearly $c_{1,2}=0$ and so $a_{1,1}c_{1,1}=m$ as well.

From the second equality $A\tp{B}=B\tp{A}$ we get the relation $b_{2,1}a_{1,1}+b_{2,2}a_{1,2}=a_{2,2}b_{1,2}$. So $b_{2,1}$ is determined from the rest of the quantities.

The number of choices for $A$ is at most $m/c_{2,2}$ given $C$; similarly
the number of choices for $B$ is at most $c_{1,1}c_{2,2}^2$. Therefore, the total number of coset representatives is at most
\begin{align}
\summ_{c_{1,1}\mid m,\, c_{2,2} \mid m}\frac{m}{c_{2,2}} c_{1,1}c_{2,2}^2 \le m^3 \sigma_0(m)^2,
\dn\end{align}
since $c_{2,1}$ is determined from $a_{1,1}c_{2,1}+a_{1,2}c_{2,2}=0$.
\end{proof}

With the lemma in hand, now we can finally say from \eqref{hnest} and \eqref{hzw2} that
\begin{align} \label{bzz1}
 \ank[2]^{-1} \mbb B^m(Z,Z) \ll (1+\rho^2)^{-k+k_0} m^{3+\epsilon}  + \mathscr S_m(Z, u \rho^{1/2}).
\end{align}
Quoting from \cite[Prop.~5]{blomer2016supnorm} it follows that for $\epsilon>0$,
\begin{equation} \label{spzcount}
\mathscr S_m(Z, u \rho^{1/2})
\ll_{\Omega,\epsilon} m^{1+\epsilon}(1+\rho^{B_1} m^{B_2})
\end{equation}
for some $B_1,B_2>0$.

Finally, with \eqref{bzz1} thus we can put together all the ingredients 
and obtain
\begin{align}
 \ank[2]^{-1} \mbb B^m(Z,Z) \ll (1+\rho^2)^{-k+k_0} m^{3+\epsilon}  + m^{1+\epsilon}(1+\rho^{B_1} m^{B_2}),
\end{align}

From \eqref{ampl2} we now see that $\mbb A(Z,Z)$ is 
\begin{align}
 & \ll_{\epsilon, \Omega} \ank[2] \Big(\sum_{r=1,2,4} \sum_{p_1 \neq p_2} (p_1 p_2)^{3(1-r)/2} |\mbb B^{p_1^r p_2^r}(Z,Z)| + \sum_{0 \le s \le 4} \sum_{p \in \mc P} p^{3-2s}  |\mbb B^{p^{2s}}(Z,Z)| \n \Big)\\
& \ll \ank[2] \left( \sum_{r=1,2,4} \sum_{p_1 \neq p_2} (p_1 p_2)^{3(1-r)/2} 
 \big( (1+\rho^2)^{-k+k_0} (p_1p_2)^{3r+\epsilon}  + (p_1p_2)^{r+\epsilon}(1+\rho^{B_1} (p_1p_2)^{rB_2}) \big) \right. \n  \\
&   \q \q \q \q \q \left.
+ \sum_{0 \le s \le 4} \sum_{p \in \mc P} p^{3-2s}  \big(
(1+\rho^2)^{-k+k_0} p^{6s+\epsilon}  + p^{2s+\epsilon}(1+\rho^{B_1} p^{2sB_2}) \big) \right) \label{bzz2}
\end{align}

We now choose $\rho$ such that $\rho^{B_1} (2L)^{8B_2}=1$, and $L=L_0$ such that $(1+\rho^2)^{-k}$ decays exponentially:
\begin{equation}
2L_0: = \lfloor k^{B_1/(16B_2) - \epsilon} \rfloor; \q \rho = (2L_0)^{-8B_2/B_1}\asymp k^{-1/2 + \epsilon}
\dn\end{equation}
With these choices \eqref{bzz2} reads
\begin{align}
\ank[2]^{-1} \mbb A(Z,Z) \ll_{\epsilon, \Omega}   &  \sum_{r=1,2,4} \sum_{p_1 \neq p_2} (p_1 p_2)^{3(1-r)/2} \big( (p_1p_2)^{3r+\epsilon} e^{-k^\epsilon}+ (p_1p_2)^{r+\epsilon} \big) \n \\
& + \sum_{0 \le s \le 4} \sum_{p \in \mc P} p^{3-2s}  \big(  p^{6s+\epsilon} e^{-k^\epsilon} + p^{2s+\epsilon} \big)
\ll_{\epsilon} \frac{L_0^{3+\epsilon}}{\log(L_0)}. \label{ampl3}
\end{align}
Comparing \eqref{ampl1} and \eqref{ampl3} we find that (with $\ank[2] \ll k^3$)
\begin{align}
 \frac{L_0^5}{(\log L_0)^2} \det(Y)^k |F_0(Z)|^2 \ll \mbb A(Z,Z) \ll_{\epsilon, \Omega} k^3 \frac{L_0^{3+\epsilon}}{\log(L_0)}
\dn\end{align}
which finally gives us
\begin{equation}
\sup_{Z \in \Omega} \det(Y)^k |F_0(Z)|^2 \ll_{\epsilon, \Omega} k^3 \frac{\log(L_0)}{L_0^{2-\epsilon}} \ll k^{3- \omega+\epsilon}
\dn\end{equation}
for some absolute constant $\omega>0$ since $B_1, B_2 $ can be effectively computed. Summarising, we have proved

\begin{thm} \label{amplithm}
Let $F \in S^2_k$, $\norm{F}_2=1$ and $F$ be an eigenfunction of all Hecke operators. Assume further that $F$ is orthogonal to the space of Saito-Kurokawa lifts. Then
$\oldnorm{F}_\Omega \ll_{\Omega} k^{3/2-\eta}$ for some absolute constant $\eta>0$.
\end{thm}

\begin{rmk}
It is of course possible to deduce a polynomial bound for $\bkzz$ by estimating the set $\mbb S_1(Z, \delta_0)$ in the non-compact setting. We plan to come back to this point in the future, motivated by section~\ref{n=1exam}.
\end{rmk}

\bibliographystyle{amsalpha}


\raggedbottom

\end{document}